\documentclass{article}

\usepackage[utf8]{inputenc}
\usepackage[T1]{fontenc}
\usepackage{amssymb}
\usepackage[english]{babel}
\usepackage{amsmath}
\usepackage{amsthm}
\usepackage{epsfig}
\usepackage{mathrsfs}
\usepackage{xcolor}
\usepackage{graphicx}
\usepackage{enumerate}
\usepackage{cases}
\usepackage{dsfont}
\usepackage[toc,page]{appendix}
\usepackage{hyperref}

\makeatletter
\newtheorem*{rep@theorem}{\rep@title}
\newcommand{\newreptheorem}[2]{%
\newenvironment{rep#1}[1]{%
 \def\rep@title{#2 \ref{##1}}%
 \begin{rep@theorem}}%
 {\end{rep@theorem}}}
\makeatother

\def\RR{\mathbb{R}}

\def\div{\mathrm {div}}

\def\L{\mathcal L}

\def\tto{\longrightarrow}

\def\pa{\partial}
\def\na{\nabla}
\def\eps{\varepsilon}
\def\vphi{\varphi}
\def\bO{\overline{\Omega}}
\def\dO{\pa\Omega}

\def\Dels{\big(-\Delta\big)^s}

\def\wpsi{\widetilde \psi}
\def\wvphi{\widetilde \vphi}

\def\bF{F_1}

\def\d{\,{\rm{d}}} 

\def\oTe{\overline{T}^\eps}

\def\PV{\mathrm{P.V.}}

\theoremstyle{plain}
\newtheorem{theorem}{Theorem}[section]
\newtheorem{lemma}[theorem]{Lemma}
\newtheorem{proposition}[theorem]{Proposition}
\newtheorem{definition}[theorem]{Definition}

\newtheorem{remark}[theorem]{Remarks}
\newtheorem{corollary}[theorem]{Corollary}

\title{Fractional diffusion limit for a kinetic equation in the upper-half space with diffusive boundary conditions}
\date{}
\author{
L. Cesbron\thanks{Centre de Math\'ematiques Laurent Schwartz,
Ecole Polytechnique, France. Partially supported by a public grant as part of the FMJH}, 
A. Mellet\thanks{Department of Mathematics and CSCAMM, University of Maryland,  USA. Partially supported by NSF Grant DMS-1501067.}, 
M. Puel\thanks{Laboratoire J.-A. Dieudonn\'e,
Universit\'e de Nice Sophia-Antipolis, France.}
} 

\makeindex
\begin{document}
\maketitle

\begin{abstract}
We investigate the fractional diffusion approximation of a kinetic equation set in the upper-half space with diffusive reflection conditions at the boundary.
In an appropriate singular limit corresponding to small Knudsen number and long time asymptotic, we derive a fractional diffusion equation with a nonlocal  Neumann boundary condition for the density of particles. Interestingly,
this asymptotic equation is different from the one derived by L. Cesbron in \cite{Cesbron18} in the case of specular reflection conditions at the boundary and does not seem to have received a lot of attention previously.
\end{abstract}

\tableofcontents

\section{Introduction}

The purpose of this paper is to investigate the fractional diffusion approximation of a linear kinetic equation set on a bounded domain with diffusive boundary conditions. The starting point of our analysis is the following linear Boltzmann equation
\begin{equation} \label{eq:kin0}
\left\{ \begin{aligned}
& \pa_t f +  v\cdot \na_x f = Q(f) \qquad &\mbox { for } (t,x,v)\in(0,\infty)\times\Omega\times\RR^N \\
& f (0,x,v) = f_{in} (x,v) &\mbox{ for } (x,v)\in\Omega\times\RR^N
\end{aligned} \right. 
\end{equation}
where $\Omega$ is a subset of $\RR^N$ and $Q$ is the linear Boltzmann Operator
\begin{align*} %\label{eq:defQ} 
Q(f) (v)&  = \int_{\RR^N} \sigma(v,w)\big[ F(v) f(w) - F(w) f(v)\big]  \d w\\
& = K(f)(v)-\nu(v) f(v).
\end{align*}
Throughout this paper, the thermodynamical equilibrium $F(v) = \tilde F(|v|^2)\geq 0$ will be a normalized heavy-tail distribution function satisfying
\begin{equation}\label{eq:F0}
\int_{\RR^N} F(v)\d v=1,\qquad  F(v)\sim \frac{\gamma}{|v|^{N+2s}} \text{ as } |v|\to\infty, \qquad s\in (1/2,1)
\end{equation}
and, to avoid unnecessarily complicated notations in the proof, we will assume that the cross section $\sigma(v,w)$ is constant equal to $\nu_0$ throughout the rigorous part of the paper.
However, the result holds without modifications if we assume instead that $\sigma(v,w)$ is 
 bounded above and below and symmetric:
$$ 0<\sigma_0 \leq \sigma(v,w) \leq \sigma_1 , \qquad \sigma(v,w)=\sigma(w,v) \quad\mbox{ for all } (v,w)\in\RR^N\times\RR^N$$
and if  the collision frequency $\displaystyle \nu (v)= \int_{\RR^N} \sigma(v,w) F(w) \d w $ satisfies  $\nu(v)\to \nu_0$ as $|v|\to\infty$.
%has a finite limit for large $|v|$:
%\begin{align} \label{def:nu0}
%\lim_{|v|\to +\infty} \nu(v) = \lim_{|v|\to +\infty} \int_{\RR^N} \sigma(v,w) F(w) \d w = \nu_0 < \infty .
%\end{align}   
%}

This kinetic equation models the evolution of a particle distribution function $f(t,x,v)\geq 0$ depending on the time $t>0$, the position $x\in\Omega$ and the velocity $v\in\RR^N$. The left hand side of \eqref{eq:kin0} models the free transport of particles, whereas the operator $Q$ in the right hand side models the diffusive and mass preserving interactions between the particles and the background. 

The equation must be supplemented by boundary conditions on $\pa\Omega$.
In this paper, we consider {\bf diffusive reflection conditions}, which can be written as: 
\begin{equation} \label{eq:BCdiff}
\gamma_-f(t,x,v) = \mathcal B [\gamma_+ f] (t,x,v) \qquad \forall (x,v)\in \Sigma_-
\end{equation}
where $\gamma_\pm f$ is the restriction of the trace $\gamma f$ on $\Sigma_{\pm} := \lbrace (x,v)\in\dO\times\RR, \, \pm n(x) \cdot v >0\rbrace$ with $n(x)$ the outward unit normal vector. 
To avoid the need of boundary layer analysis, we assume that the boundary operator $\mathcal{B}$ takes the form
\begin{equation}\label{eq:B}
\mathcal B [g](x,v):= 
\alpha_0 F(v) \int_{w\cdot n(x)>0} g (x,w) |w\cdot n(x)| \d w \qquad \forall (x,v)\in \Sigma_-
\end{equation}
with the same $F(v)$ as in \eqref{eq:F0} and with $\alpha_0$ a normalization constant chosen such that
$$ \alpha_0 \int_{v\cdot n<0} |v\cdot n|  F(v)\d v = 1$$
for any unit vector $n$ (this integral is well defined since $F(v)\sim \frac 1 {|v|^{N+2s}}$ and $ s > 1/2$).

The diffusion approximation of such an equation is obtained by investigating the long time, small mean-free-path asymptotic behavior of $f$. To this end we introduce the Knudsen number $\eps$ and the following rescaling of \eqref{eq:kin0}-\eqref{eq:BCdiff}
\begin{equation}\label{eq:kin}
\left\{ \begin{aligned}
& \eps^{2s-1} \pa_t f^\eps+  v\cdot \na_x f^\eps = \frac{1}{\eps} Q(f^\eps) \qquad &\mbox{ for } (t,x,v)\in(0,\infty)\times\Omega\times\RR^N, \\
& f^\eps(0,x,v) = f_{in} (x,v) &\mbox{ for } (x,v)\in\Omega\times\RR^N,\\[5pt]
&\gamma_-f^\eps(t,x,v) = \mathcal B [\gamma_+ f^\eps] (t,x,v) & \mbox{ for } (t,x,v) \in (0,\infty)\times\Sigma_-.
\end{aligned} \right. 
\end{equation}
We see that the particular choice of power of $\eps$ in front of the time derivative in \eqref{eq:kin} depends on the equilibrium $F$. The correct scaling was established in \cite{MelletMischlerMouhot11} (see also \cite{AcevesMellet,Mellet10,AbdallahMelletPuel11}) where it was shown that if $\Omega$ is the whole space $\RR^N$ then the solution $f^\eps$ of \eqref{eq:kin} converges, as $\eps$ goes to zero, to a function
\begin{align*}
f^0(t,x,v)= \rho(t,x)F(v) \in \ker Q % := \lbrace \rho(t,x) F(v);\, \rho : (0,\infty)\times\RR^N \rightarrow \RR^N \rbrace
\end{align*}
where $\rho(t,x)$ solves the following fractional diffusion equation:
$$
\left\{ \begin{aligned}
& \pa_t \rho + \kappa \Dels \rho = 0 \qquad &\mbox{ for } (t,x)\in (0,+\infty)\times\RR^N ,\\
& \rho(0,x) = \rho_{in}(x) = \int_{\RR^N} f_{in}(x,v) \d v &\mbox{ for } x\in\RR^N
\end{aligned} \right. 
$$
for some $\kappa>0$.
Recall that the fractional Laplacian $\Dels$ is a non-local integro-differential operator which can be defined through its Fourier transform:
\begin{align*}
\mathcal{F} \left( \Dels \rho \right) (\xi) :=  |\xi|^{2s} \mathcal{F} \left( \rho \right) (\xi) 
\end{align*}
or equivalently as a singular integral
\begin{align*}
\Dels \rho(x) = c_{N,s} \PV \int_{\RR^N} \frac{\rho(x)-\rho(y)}{|x-y|^{N+2s}} \d y 
\end{align*}
where $c_{N,s}$ is an explicit constant, see e.g. \cite{DiNezzaPalatucciValdinoci12,Kwasnicki15} for more details. 
\medskip 

In this paper, though, the equation is set in a  subset $\Omega$ of $\RR^N$. So we expect to derive a fractional diffusion equation confined to the domain $\Omega$. 
The question at the heart of this paper is to determine the appropriate boundary conditions for this asymptotic equation.
When the thermodynamical equilibrium $F$ is a Gaussian (or Maxwellian) distribution then it is well known that the diffusion limit of \eqref{eq:kin}, with $s=1$, leads to the classical heat equation 
supplemented with 
homogeneous Neumann boundary conditions.
Interestingly, these boundary conditions are not very  sensitive to the type of microscopic boundary conditions. In particular, if  instead of \eqref{eq:BCdiff}, we supplement equation \eqref{eq:kin0} with specular reflection conditions \eqref{eq:SR}, or with a combination of diffuse and specular reflections (Maxwell boundary conditions), the limiting boundary conditions are the same homogeneous Neumann boundary conditions mentioned above.
\medskip

However, the issue of boundary condition is much more delicate with nonlocal operators such as  fractional Laplacians. 
Indeed, these operators are classically associated with alpha-stable L\'evy processes (or jump processes). Unlike the Brownian motion, these processes are discontinuous and may exit the domain without touching the boundary.
This is the reason why the usual Dirichlet problem for the fractional Laplacian requires a prescribed data in $\RR^N\setminus \overline \Omega$ rather than just on the boundary $\pa \Omega$ \cite{FelsingerKassmannVoigt15}.
Neumann boundary value problems correspond to processes that are not allowed to jump outside $\Omega$ (sometimes referred to as censored stable processes).
Several constructions of such processes are possible. 
A classical construction consist in cancelling the process after any outside jump and restarting  it  at the last position inside the set (resurrected processes).
This construction (see \cite{BogdanBurdzyChen03,GuanMa05,GuanMa06} for details) leads to the {\it regional fractional laplacian} defined by
\begin{align*}
\Dels_\Omega \rho(x) = c_{N,s} \PV \int_{\Omega} \frac{\rho(x)-\rho(y)}{|x-y|^{N+2s}} \d y 
\end{align*}
However, other constructions of censored processes are possible.
Because of the nonlocal nature of the problem, the choice of boundary conditions for the underlying process typically changes the operator inside the domain.
In \cite{Barles14}, several such operators are discussed.  
For instance the process that reaches a position $y\notin \Omega$ 
can be restarted inside $\Omega$ by projecting $y$ onto $\pa\Omega$, or by reflecting $y$ about $\pa\Omega$ (see discussion below).
In \cite{DiPierroRosotonValdinoci17} a different Neumann problem is obtained by restarting the process from a point $x\in \Omega$ chosen randomly with probability proportional to $|x-y|^{-N-2s}$.

\medskip

In a recent paper \cite{Cesbron18}, L. Cesbron  studied 
the derivation of fractional diffusion approximation from a kinetic model  in a bounded domain 
with 
{\bf specular reflection} at the boundary.
These conditions read:
\begin{align} \label{eq:SR}
\gamma_- f^\eps (t,x,v) = \gamma_+ f^\eps ( t,x, \mathcal{R}_x v), \quad \mathcal{R}_x (v) = v - 2 \big( n(x) \cdot v\big) n(x), \quad (t,x,v) \in (0,T)\times\Sigma_-.
\end{align}
In that case, the asymptotic 
equation reads
$$
\left\{ \begin{aligned}
&\pa_t \rho + (-\Delta)_{\rm{\tiny{SR}}}^s \rho = 0 \quad & \mbox{ for } (t,x)\in (0,+\infty)\times\Omega \\
&\rho(0,x) = \rho_{in}(x) &\mbox{ for } x\in \Omega
\end{aligned} \right. 
$$
where 
\begin{equation} \label{def:DelsSR}
(-\Delta)_{\text{\rm{SR}}}^s\rho(x) = c_{N,s} \PV \underset{\RR^N}{\int} \frac{\rho(x) - \rho\big(\eta(x,w)\big)}{|w|^{N+2s}} \d w
\end{equation}
where $\eta: \Omega\times\RR^N \to \bO$ is the flow of the free transport equation with specular reflection on the boundary. When $\Omega$ is the upper-half space, we simply have
$$\eta(x,w) = \begin{cases}
x+w & \mbox{ if } x_N+w_N>0 \\
(x'+w',-x_N-w_N)& \mbox{ if } x_N+w_N<0 
\end{cases}
$$
and the underlying alpha stable process is the process which is moved back inside $\Omega$ by a mirror reflection about the boundary $\pa\Omega$ upon leaving the domain (see \cite{Cesbron18,Barles14}).

\medskip

Our main result in this paper states that when the boundary conditions at the microscopic level
are given by \eqref{eq:BCdiff}, then the asymptotic operator is
$$ \L  [\rho]  = - c_{s,N}\PV\int_\Omega \na \rho(y) \cdot\frac{y-x}{|x-y|^{N+2s}}\, dy$$
which is neither the regional fractional Laplacian, nor the operator \eqref{def:DelsSR} (see \eqref{eq:opL3} for the precise relation between $ \L$ and $ (-\Delta)^s_\Omega$).
Furthermore, this operator can be written in divergence form as $\div D^{2s-1}[\rho]$
where $D^{2s-1}[\rho]$ is a nonlocal gradient of order $2s-1$ (see \eqref{eq:D0}), and the fractional diffusion equation must be supplemented by the following Neumann type condition
$$  D^{2s-1}[\rho] \cdot n=0 \qquad \mbox{ on } \pa\Omega$$
(see \eqref{eq:fracNeumann}).
Note that while the operator $D^{2s-1}$ is non local, the boundary condition itself is only assumed to hold on the boundary $\pa\Omega$.
This is thus different from the Nonlocal Neumann problem studied in \cite{DiPierroRosotonValdinoci17}, where the Neumann condition is set in $\RR^N\setminus \Omega$.

The main takeaway from this paper is thus that 
for the fractional diffusion approximation, the limiting operator is very sensitive to the particular choice of microscopic boundary conditions.
Note also that unlike \eqref{eq:SR} where the interaction with the boundary is entirely included in the diffusion operator, here the diffusive boundary condition \eqref{eq:BCdiff} gives rise to the boundary condition above. This can be seen as a result of the difference in nature of the kinetic boundary conditions: the local-in-velocity specular reflection vs. non-local-in-velocity diffusive condition. 
\medskip

The goal of this paper is to formally explain the derivation of the asymptotic equations for \eqref{eq:kin} in convex subsets of $\RR^N$ and to rigorously prove this derivation  when $\Omega$ is the upper half-space.

\subsection{Main results and outline of the paper}

The existence of solutions to equation \eqref{eq:kin} is a delicate problem because it is difficult to control the trace $\gamma_+ f$ in an appropriate functional space (see \cite{Mischler1,Mischler2}).
Note that for a  given test function  $\phi\in\mathcal D ([0,\infty)\times\overline\Omega\times\RR^N)$,  
a smooth solution of \eqref{eq:kin} will satisfy
\begin{align*}
& - \underset{\RR^+\times\Omega\times\RR^N}{\iiint} f^\eps \Big(  \pa_t \phi +\eps^{1-2s} v \cdot \na_x\phi \Big) \d v \d x \d t \\
&+ \eps^{1-2s} \underset{\RR^+\times \Sigma_+}{\iint} \gamma_+ f^\eps  \bigg( \gamma_+ \phi - \mathcal B ^*[\gamma_- \phi] \bigg) |v\cdot n|\d v\d S(x) \d t\\
& =\eps^{-2s}\underset{\RR^+\times\Omega\times\RR^N}{\iiint} f^\eps Q^*(\phi) \d v\d x\d t+ \underset{\Omega\times\RR}{\iint} f_{in} (x,v) \phi(0,x,v) \d  x \d v.
\end{align*}
where
$$ 
\mathcal B ^* [\gamma_- \phi ] (x,v)= \int_{w\cdot n(x)<0} \alpha_0 F(w) \gamma_-\phi(w) |w\cdot n(x)|\d w.$$

A classical way of defining weak solutions of \eqref{eq:kin}, \eqref{eq:BCdiff} without having to deal with the trace $\gamma f$ is then the following (see for instance \cite{MelletVasseur}):
\begin{definition}\label{def:weak}
We say that a function $f(t,x,v)$ in $L^2_{F^{-1}}((0,\infty)\times \Omega\times\RR^N)$ is a weak solution of \eqref{eq:kin}  if for every test functions  $\phi(t,x,v)$ such that 
$\phi$, $\pa_t \phi$ and $v\cdot \na_x \phi$ are $L^2_{F}((0,\infty)\times \Omega\times\RR^N)$
and satisfying the boundary condition
$$\gamma_+ \phi = \mathcal B ^*[\gamma_- \phi] ,$$
the following equality holds:
\begin{align}
&- \underset{\RR^+\times\Omega\times\RR^N}{\iiint} f^\eps \Big(  \pa_t \phi +\eps^{1-2s} v \cdot \na_x\phi \Big) \d v\d x\d t \nonumber \\
& \qquad\quad =\eps^{-2s}\underset{\RR^+\times\Omega\times\RR^N}{\iiint} f^\eps Q^*(\phi) \d v\d x\d t + \underset{\Omega\times\RR}{\iint} f_{in} (x,v) \phi(0,x,v) \d x\d v.\label{eq:weak0}
\end{align}
\end{definition}
Here and in the rest of the paper, we used the notation
$$ L^2_{F^{-1}} ((0,\infty)\times \Omega\times\RR^N)= \left\{f(t,x,v) \, ;\, \int_0^\infty\int_\Omega\int_{\RR^N} |f(t,x,v)|^2 \frac{1}{F(v)}\, dv\, dx\, dt<\infty\right\}$$
and a similar definition for $L^2_{F}((0,\infty)\times \Omega\times\RR^N)$.
\medskip

In order to write our main  result, we now define the operator 
\begin{equation}\label{eq:D0} 
D^{2s-1}[u](x) = \gamma \nu_0^{1-2s} \Gamma(2s-1) \int_\Omega (y-x)\cdot \na u(y)  \frac{y-x}{|y-x|^{N+2s}} \, dy
\end{equation}
which is defined pointwise for example if $\na u \in L^\infty_{loc}(\Omega)\cap L^1(\Omega)$ 
(note that we included the constant $\gamma \nu_0^{1-2s} $ which depends $F$ and $\nu$ in this definition in order to simplify the notations later on). In particular, if $N=1$ and $\Omega=\RR$, we find
\begin{align*}D^{2s-1}[u](x) & = \gamma \nu_0^{1-2s} \Gamma(2s-1) \int_\RR \frac{u'(y)}{|x-y|^{N-2(1-s)}}  \, dy\\
& = \gamma \nu_0^{1-2s}  c (-\Delta)^{-(1-s)} u'(x),
\end{align*}
for some constant $c$.
So the operator $D^{2s-1}$ can be interpreted as a fractional   gradient of order $2s-1\in(0,1)$

Our main result is then the following:
\begin{theorem}\label{thm:main}
Assume that $Q$ is given by \eqref{eq:Q0} 
 and that $F$ satisfies \eqref{eq:F} with $s\in(1/2,1)$. Let $\Omega$ be the upper half space 
$$ \Omega = \{x\in \RR^N\,;\, x_N >0\}.$$
Assume that $f^\eps(t,x,v)$ is a weak solution of \eqref{eq:kin} in $(0,\infty)\times \Omega\times\RR^N$
in the sense of Definition~\ref{def:weak} and satisfies the energy inequality \eqref{eq:energy}. 
Then, up a subsequence, the function $f^\eps(t,x,v)$ converges weakly in $L^\infty(0,\infty; L^2_{F^{-1}} (\Omega\times\RR^N))$, as $\eps$ goes to $0$, to a function $\rho(t,x) F(v)$ where $\rho(t,x)$ satisfies
\begin{align} \label{eq:weakFracNeumann}
\underset{\RR^+\times\Omega} {\iiint} &\rho(t,x) \Big(  \pa_t \psi(t,x) +\div D^{2s-1} [\psi](t,x) \Big) \d t\d x + \underset{\Omega}{\iint} \rho_{in} (x) \psi(0,x) \d x =0 
\end{align}
for all test function $\psi\in W^{1,\infty}(0,\infty;H^2(\Omega))$, such that $ \div D^{2s-1} [\psi] \in L^{2}(\RR_+\times\Omega)$ and  
\begin{equation}\label{eq:psineumann}
 D^{2s-1} [\psi] \cdot n =0 
 %\quad\mbox{ and } \quad 
 %\lim_{\delta\to 0 } \delta^{-2(2s-1)}\underset{\RR_+\times\RR_+^N }{\iint}|  \pa_{x_N} \psi(t,x)|^2 1_{x_N\leq \delta} \d t\d x =0.
\end{equation}
\end{theorem}
We now make several remarks concerning this result:
\begin{enumerate}
\item As mentioned in the introduction, the result holds for more general collision operators $Q$. We restrict ourselves to the simplest case here in order to focus on the novelty of our analysis, which is to deal with the boundary conditions. 
\item Equation \eqref{eq:weakFracNeumann} is the fractional equivalent of the following  weak formulation of the usual heat equation with Neumann boundary conditions:
$$
\begin{cases}
\displaystyle \underset{\RR^+\times\Omega} {\iiint} \rho \Big(  \pa_t \psi(t,x) +\Delta  \psi (t,x) \Big) \d t\d x + \underset{\Omega}{\iint} \rho_{in} (x) \psi(0,x) \d x =0  \\
\displaystyle \mbox{for all } \psi \in W^{1,\infty}(0,\infty;H^2(\Omega)) \mbox{ such that } \na \psi(x) \cdot n(x)=0 \mbox{ on } \pa\Omega.
\end{cases}
$$
In particular, condition \eqref{eq:psineumann} is  the nonlocal equivalent of this classical Neumann boundary condition.

%for all test function $\psi$ satisfying
%the Neumann boundary condition $\na \psi(x) \cdot n(x)=0$ for all $x\in\pa\Omega$ (this last condition is here replaced by  ).

\item Using the following integration by parts formula (which we will prove in Proposition \ref{prop:ipp}):
$$
\int_{\Omega} \div D^{2s-1}[\vphi] \psi \d x - \int_{\Omega} \vphi \, \div D^{2s-1} [\psi]\d x  = \int_{\pa \Omega} \left[\psi D^{2s-1}[\vphi]\cdot n -
 \vphi D^{2s-1}[\psi]\cdot n \right]\d S(x)
$$
we see that Equation \eqref{eq:weakFracNeumann} is the weak formulation for 
 the following fractional Neumann boundary problem: 
\begin{equation} \label{eq:fracNeumann}
\left\{ \begin{aligned}
& \pa_t \rho - \div D^{2s-1} [\rho] = 0 \qquad \mbox{ in } (0,\infty)\times \Omega\\
& D^{2s-1} [\rho] \cdot n = 0 \qquad \mbox{ in } (0,\infty)\times\pa  \Omega\\
& \rho(0,x) = \rho_{in}(x) \qquad \mbox{ in } \Omega.
\end{aligned} \right. 
\end{equation}
%in the sense that if $\rho\in C^1(0,\infty;C^2(\overline \Omega))$ satisfies \eqref{eq:weakFracNeumann} for all test functions in $W^{1,\infty}(0,\infty;H^2(\Omega))$, satisfying
%\eqref{eq:psineumann}, then $\rho(t,x)$ is a classical solution of \eqref{eq:fracNeumann}.

\item We need to require that $ \div D^{2s-1} [\psi] \in L^{2}(\RR_+\times\Omega)$ in Theorem \ref{thm:main},
because such a fact is not implied by the condition $\psi\in W^{1,\infty}(0,\infty;H^2(\Omega))$ (which might seem surprising if one thinks of $\div D^{2s-1}$ as a Laplacian of order $s\in(1/2,1)$).
We will characterize precisely in Proposition \ref{prop:L2bd} the functions such that $ \div D^{2s-1} [\psi] \in L^{2}(\RR_+\times\Omega)$ and in particular, we will prove that  when $s\geq 3/4$, this condition requires $\psi$ to  satisfy the local Neumann boundary condition $\na \psi \cdot n = 0 $ on $\pa\Omega$. %owever, we will see that this condition is very natural. Indeed, we prove in Proposition \ref{prop:L2bd} that if $\psi \in H^{2s+\beta}(\Omega)$ for some $\beta>0$ and $\L[\psi]\in L^2(\Omega)$, then this condition is satisfied.
This suggests that solutions of \eqref{eq:fracNeumann} also satisfy the classical Neumann boundary conditions at the boundary, though this fact emerges as a consequence of the regularity theory, rather than as a boundary condition necessary to get a unique solution.

\item As explained in the first part of this introduction, we will show that the main operator in \eqref{eq:fracNeumann}  is 
\begin{equation}\label{eq:opL2}
\L[\rho](x)  : = \gamma \nu_0^{1-2s}\Gamma(2s)\PV \int_{\Omega}  \na \rho (y) \cdot  \frac{y-x}{|y-x|^{N+2s}}    \d y .
 \end{equation}
Indeed, taking the divergence in \eqref{eq:D0}, we obtain (formally at least)
$$
\div D^{2s-1}[\rho]= \mathcal L[\rho] .$$
We will rigorously justify this formula later on, see Lemma~\ref{lem:opL2}.
We see in particular that when $\Omega=\RR^N$, we recover the usual fractional Laplacian of order $s$ in $\RR^N$ (up to a constant).
\end{enumerate}

\medskip
\medskip

Because equation \eqref{eq:fracNeumann}
does not seem to have been studied in details before, we will prove the following theorem:
\begin{theorem}\label{thm:evolution}
For all $\rho_{in}\in L^2(\Omega)$, the evolution problem
\begin{equation}\label{eq:evolution}
\left\{
\begin{array}{ll}
\pa_t \rho - \div D^{2s-1}[\rho]=0 & \mbox{ in } (0,\infty)\times \Omega,\\
D^{2s-1}[\rho]\cdot n = 0 & \mbox{ on } (0,\infty)\times\pa\Omega, \\
\rho(0,x) = \rho_{in}(x) & \mbox{ in } \Omega.
\end{array}
\right.
\end{equation}
has a unique  solution $\rho \in C^0(0,\infty;L^2(\Omega))\cap L^2(0,\infty;D(\mathcal L))$ where
$$
D(\mathcal L)=\{\vphi \in H^s(\Omega)\,;\, \mathcal L[\vphi] \in L^2(\Omega), \quad D^{2s-1}[\vphi]\cdot n=0\mbox{ on } \pa\Omega \}.
$$
\end{theorem}

However, we do not show, in this paper, that 
the function $\rho(t,x)$ identified in Theorem \ref{thm:main} is the unique solution of \eqref{eq:evolution} (note that, once proved, such a  uniqueness result implies that the whole sequence $f^\eps$, and not just a subsequence, converges to $\rho F$). 
To prove such a fact requires additional regularity results for the solutions of \eqref{eq:evolution}.
Namely, we need the weak solution of \eqref{eq:evolution} - or rather that of the dual problem - to be  in $W^{1,\infty}(0,\infty;H^2(\Omega))$ for smooth initial data.
This is actually a delicate problem which requires a detailed  analysis of the boundary regularity of the solution of \eqref{eq:evolution} and which does not seem to have been addressed so far in the literature.
It is the object of the companion paper  \cite{CMP}.

%We will also see that the operator $\L$ is associated to the Dirichlet form
%\begin{align*}
%\mathcal E(u) & = \int_\Omega D^{2s-1}[u] \cdot \na u\, dx \\
%& = \int_\Omega \int_\Omega \frac{ [(y-x) \cdot \na u(x)] [(y-x) \cdot \na u(y)]}{|x-y|^{N+2s}}\, dx\, dy
%\end{align*}

%bilinear form

%Dirichlet form

%uniqueness in 1D?

%\medskip
%\medskip

%Furthermore, we prove 
% an integration by parts formula that will justify the fact
% that \eqref{eq:weakFracNeumann} is indeed the natural weak formulation of \eqref{eq:fracNeumann} and conclude this paper with a proof a well-posedness of \eqref{eq:fracNeumann}. More precisely we first consider the stationary equation
%\begin{equation}\label{eq:Neumann}
%\left\{
%\begin{array}{ll}
%\vphi (x)- \mathcal L [\vphi](x)= g(x) & \mbox{ for all  } x\in \Omega, \\
%D^{2s-1} [\vphi](x) \cdot n(x) =0 & \mbox{ for all } x\in \pa\Omega.
%\end{array}
%\right.
%\end{equation}
%Using a Lax-Milgram argument we prove that for $g \in L^2(\Omega)$ this equation has a unique solution in the functional space

%As a result, the Hille-Yoshida theorem will give the following 

%Identifying  In this paper, we prove uniqueness of solution of \eqref{eq:evolution} in the naturally associated functional space $\mathcal{D}(\L)$, and that $f^\eps$ converges up to a subsequence to a solution of \eqref{eq:evolution} in the sense given in Theorem \ref{thm:main} and we will dwell on further identifications in future works. 

\medskip

Finally, we need to stress that we only rigorously prove the fractional diffusion approximation when $\Omega$ is the upper half-space because in this case the boundary values do not interact with each other via the boundary conditions which simplifies some of the arguments in the (already delicate) proof. However the result certainly holds for general convex domains. 
%However in that case the operator $D^{2s-1}$ on the boundary will infer a relation between the boundary values and as a result the approximate test functions we construct in Section \ref{sec:approxtest} and the control of the resulting errors in Section \ref{sec:cor} would be sensibly more difficult. 

\medskip

%{\color{blue} 
%Comments on the limit problem: comparison with \eqref{eq:specdiff}, comments on why we can't do the proof is a general domain (interaction between the boundary values and consequences on the construction of $\psi_\eps$) and also mention that we have not proved that the limit of $f^\eps$ is the unique solution of \eqref{eq:fracNeumann}, i.e. convergence only up to a subsequence.
%}

{\bf{Outline of the paper}.} 
The rest of the paper is organized as follows: 
In the second part of this introductory section 
we will briefly present  the main ideas of the proof.
Section \ref{sec:pre}, is devoted to some preliminary results: First we recall some important properties of the solutions of the kinetic equation \eqref{eq:kin}, in particular the existence of weak solutions  and the convergence to a thermodynamical equilibrium.
We also establish (in Section \ref{sec:DL}) some important properties of the operators $D^{2s-1}$ and  $\L = \div D^{2s-1} $, some of which are needed for the proof of our main result, as well as others that are of independent interest.
Section \ref{sec:rig}, is devoted to the proof of the main result, Theorem \ref{thm:main}. 
%will first construct appropriate test functions and prove the convergence to \eqref{eq:fracNeumann} when $\Omega$ is the upper half-space, with the assumption that the equilibrium $F$ has the same decay as the equilibrium of the fractional Fokker-Planck operator studied in \cite{Cesbron18}. 
Finally, in Section \ref{sec:Lwellposed}
 we study the asymptotic fractional Neumann problem \eqref{eq:fracNeumann} and prove Theorem~\ref{thm:evolution}.

\subsection{Idea of the proof}\label{sec:formal}
In this section we explain the main idea of the proof. As in previous works on this topic, e.g. \cite{Mellet10,AbdallahMelletPuel11+,AcevesSchmeiser17}, 
for a given test function $\psi(t,x)$ defined in $[0,\infty)\times\Omega$, we introduce $\phi^\eps$ solution of the auxiliary problem
\begin{equation}\label{eq:Phieps}
\nu \phi^\eps - \eps v \cdot \na_x\phi^\eps  = \nu \psi \qquad \mbox{ in } \Omega\times\RR^N.
\end{equation}
When $\Omega=\RR^N$ this equation can easily be solved explicitly. In our framework, this transport equation must  be supplemented with the boundary condition
\begin{equation} 
\gamma_+ \phi^\eps  (t,x,v) = \mathcal B^* [\gamma_- \phi^\eps](t,x,v)\qquad (x,v)\in  \Sigma_+.\label{eq:Phibceps}
\end{equation}
Assuming that we can find such a function $\phi^\eps$, we note that 
since $\psi$ does not depend on $v$, we have $K^*(\psi)=\nu\psi $, and so
\begin{align*} 
Q^*(\phi^\eps)  +  \eps v \cdot \na_x\phi^\eps & = K^*(\phi^\eps) -\nu \phi^\eps +  \eps v \cdot \na_x\phi^\eps\\
& = K^*(\phi^\eps) -\nu \psi\\
& = K^*(\phi^\eps- \psi).
\end{align*}
Taking $\phi^\eps$ as a test function in  \eqref{eq:weak0} (which we can do since $\phi^\eps$ satisfies \eqref{eq:Phibceps}), we deduce
\begin{align*}
- \underset{\RR^+\times\Omega\times\RR^N}{\iiint}  f^\eps   \pa_t \phi \d v\d x\d t - \underset{\Omega\times\RR}{\iint} f_{in} (x,v) \phi(0,x,v)
  &  =  \eps^{-2s}\underset{\RR^+\times\Omega\times\RR^N}{\iiint} f^\eps  K^*(\phi^\eps- \psi) \d v\d x\d t\nonumber \\
 & =  \eps^{-2s}\underset{\RR^+\times\Omega\times\RR^N}{\iiint} K(f^\eps)   [  \phi^\eps -  \psi]  \d v\d x\d t.
 \end{align*}
 Next we introduce the decomposition
$$ f^\eps (t,x,v)= \rho^\eps(t,x) F(v)  + g^\eps(t,x,v), \qquad \rho^\eps(t,x) = \int_{\RR^N} f^\eps(t,x,v)\d v$$
where we expect $\| g^\eps\| \ll1$ since $f^\eps$ converges to $\ker Q$.
Using the fact that $K(F)=\nu F$, we can write
$$ K( f^\eps) = \rho^\eps K( F)  + K(g^\eps) = \rho^\eps \nu F  + K(g^\eps) .$$
We thus have
\begin{align*}
  \eps^{-2s}\underset{\RR^+\times\Omega\times\RR^N}{\iiint} K(f^\eps)   [  \phi^\eps -  \psi]  \d v\d x\d t
& =  \eps^{-2s}\underset{\RR^+\times\Omega\times\RR^N}{\iiint} \rho^\eps \nu(v) F(v)   [  \phi^\eps -  \psi]  \d v\d x\d t\nonumber\\
&\quad + \eps^{-2s}\underset{\RR^+\times\Omega\times\RR^N}{\iiint} K(g^\eps)   [  \phi^\eps -  \psi]  \d v\d x\d t.
 \end{align*}
The second term in the right hand side should converge to zero, while the first term can  be written as 
$$ \eps^{-2s}\underset{\RR^+\times\Omega\times\RR^N}{\iiint} \rho^\eps \nu(v) F(v)   [  \phi^\eps -  \psi]  \d v \d x\d t=\underset{\RR^+\times\Omega }{\iint} \rho^\eps \tilde \L^\eps[\psi] \d x\d t
$$
with (using \eqref{eq:Phieps}):
\begin{align} 
\tilde \L^\eps[\psi](x) &  := \eps^{-2s}\int_{\RR^N}\nu(v) F(v)   [  \phi^\eps(x,v) -  \psi(x)]  \d v %\label{eq:defL1} 
\nonumber
\\
&  = \eps^{-2s}\int_{\RR^N}  F(v) \eps v \cdot \na_x \phi^\eps(x,v) \d v  \nonumber% \label{eq:defL3}
\\
& = \div_x\left( \eps^{1-2s} \int_{\RR^N}   v F(v) \phi^\eps(x,v) \d v \right).\label{eq:defL2}
\end{align}

Gathering all those computations, we finally arrive at the weak formulation
\begin{align}
& - \underset{\RR^+\times\Omega\times\RR^N}{\iiint}  f^\eps   \pa_t \phi^\eps \d v \d x\d t - \underset{\Omega\times\RR}{\iint} f_{in} (x,v) \phi^\eps (0,x,v) \nonumber \\
  & \qquad\qquad \qquad  =  \underset{\RR^+\times\Omega }{\iint} \rho^\eps \tilde\L^\eps[\psi] \d x\d t 
+   \eps^{-2s}\underset{\RR^+\times\Omega\times\RR^N}{\iiint} K(g^\eps) (\phi^\eps- \psi) \d v \d x\d t\label{eq:weakeps}
 \end{align}
 
 The proof then consists in passing to the limit in this weak formulation.
Passing to the limit in the left hand side  requires $\phi^\eps$ to converge to $\psi$ strongly in some $L^2$ space, which is reasonable in view of \eqref{eq:Phieps} (note also that since $\psi$ does not depends on $v$, it trivially satisfies the boundary condition \eqref{eq:Phibceps}).
For the right hand side, we notice that the last term should vanish in the limit since $f^\eps-\rho^\eps F\to 0$, so
the main step in the proof is  to identify the limit of $\tilde \L^\eps[\psi]$ for appropriate test functions $\psi$.

When $\Omega=\RR^N$, this task is greatly simplified by the fact that equation \eqref{eq:Phieps} yields an explicit formula for $\phi^\eps$ as a function of $\psi$.
When $\Omega$ is a proper subset of $\RR^N$, the task is  more delicate.
\medskip

In order to identify the limit of $\tilde \L^\eps[\psi]$, we introduce the following operator:\begin{equation}\label{eq:Deps}
\tilde D_\eps^{2s-1}[\psi](x) := \eps^{1-2s} \int_{\RR^N}  v   F(v) [\phi^\eps(x,v)  -\psi(x)]    \d v  .
\end{equation}
With this notation, we have (using \eqref{eq:defL2} and the fact that $\int_{\RR^N} vF(v)\d v = 0$):
%, we can write the operator $\L^\eps$ from \eqref{eq:defL2} as:
\begin{align}
\tilde \L^\eps[\psi] (x)
& = \div_x\left( \eps^{1-2s} \int_{\RR^N}   v F(v)  \phi^\eps(x,v) \d v\right)\nonumber \\
& = \div_x\left( \eps^{1-2s} \int_{\RR^N}   v F(v) [ \phi^\eps(x,v) - \psi(x)]\d v\right)\nonumber\\
& = \div_x \tilde D^{2s-1}_\eps[\psi](x).\label{eq:defLsd}
\end{align}
The key step in the proof is thus to show that for appropriate test function $\psi$ we have
$$
\tilde D^{2s-1}_\eps[\psi] \longrightarrow   D^{2s-1}[\psi] \qquad\mbox{ as  } \eps\to 0
$$
where $D^{2s-1}$ is the fractional derivative (or gradient) of order $2s-1$ defined by \eqref{eq:D0},
and 
$$
\tilde \L^\eps[\psi] \longrightarrow   \div D^{2s-1}[\psi] \qquad\mbox{ as  } \eps\to 0
$$
However, it should be noted that, without further assumptions on $\psi$, the term 
$$ \underset{\RR^+\times\Omega }{\iint} \rho^\eps \tilde \L^\eps[\psi] \d x\d t . $$
in \eqref{eq:weakeps}
 should yield, in the limit, an appropriate boundary term as well.
So the convergence above will only hold "up to the boundary" if $\psi$ satisfies the following appropriate non-local Neumann boundary condition:
%$$
%D_\eps^{2s-1}[\psi](x) \cdot n(x) = 0 \qquad \mbox{ for all } x\in \pa\Omega
%$$
%which yields the the Neumann boundary condition
$$
 D^{2s-1}[\psi](x) \cdot n(x) =0\; , \qquad \mbox{ for all } x\in\pa\Omega.
$$
%, which make passing to the limit in  \eqref{eq:weakeps} problematic since we do not have much control on the trace of $\rho^\eps$ along $\pa\Omega$.
%To solve this problem, we will thus make an additional assumption on $\psi$ which will cancel this boundary term (heuristically, we need to assume that $\psi$ satisfies a non-local Neumann boundary condition). In turn, this additional assumption will allow us to write a explicit formula for $\phi^\eps$. This construction is detailed in the next section.
Assuming that all the convergences above holds, we see that passing to the limit in equation \eqref{eq:weakeps}, using the fact that $f^\eps\to \rho(t,x)F(v)$,  yields:
\begin{align*}
\underset{\RR^+\times\Omega} {\iiint} &\rho \Big(  \pa_t \psi(t,x) +   \div  D^{2s-1} [ \psi] (t,x) \Big) \d t\d x + \underset{\Omega}{\iint} \rho_{in} (x) \psi(0,x) \d x =0 %\label{eq:weakas}. 
\end{align*}
which is the main claim of Theorem \ref{thm:main}.

\section{Preliminary results}\label{sec:pre}
\subsection{Entropy inequality and existence of weak solutions for \eqref{eq:kin}} \label{sec:apriori}
We end this introduction with a short proof of the classical a priori estimates satisfied by weak solutions of \eqref{eq:kin}, and which are key in  showing the convergence of $f^\eps$ toward a thermodynamical equilibrium (the kernel of $Q$):
\begin{lemma} \label{lem:apriori}
Let $f_{in}$ be in $L^2_{F^{-1}}(\Omega\times\RR^N)$ and let $f^\eps(t,x,v)$ be a strong solution of 
\eqref{eq:kin} satisfying the boundary condition \eqref{eq:BCdiff}.
Then $f^\eps$ satisfies
\begin{equation}\label{eq:energy}
\| f^\eps (t)\|^2_{L^2_{F^{-1}}(\Omega\times\RR^N)} +\eps^{-2s} \int_0^t\|f^\eps(s)-\rho^\eps(s) F\|^2_{L^2_{F^{-1}}(\Omega\times\RR^N)}\, ds \leq \| f_{in} (t)\|^2_{L^2_{F^{-1}}(\Omega\times\RR^N)}
\end{equation}
for all $t\geq 0$.
\end{lemma}

\begin{proof}
Multiplying \eqref{eq:kin} by $f^\eps/F$ and integrating with respect to $x$ and $v$ we get 
\begin{align*}% \label{eq:apriori}
\eps^{2s-1} \frac{\d}{\d t} \iint_{\Omega\times\RR^N} |f^\eps|^2 \frac{\d x \d v }{F(v)} + \iint_{\Sigma} |\gamma f^\eps|^2 v\cdot n(x) \frac{\d S(x) \d v}{F(v)}  = \frac{1}{\eps} \iint_{\Omega\times\RR^N} f^\eps Q(f^\eps) \frac{\d x \d v}{F(v)}.
\end{align*}
Inequality \eqref{eq:energy} thus follows from the following classical inequality (see for instance \cite[Lemma A.1]{Mellet10}):
$$
- \int_{\RR^N}f(v)Q(f)(v)\frac{dv}{F(v)} \geq \int_{\RR^N} \frac{|f(v)-\rho F(v)|^2}{F(v)} \, dv \qquad \forall f\in L^2_{F^{-1}}(\RR^N), \quad \rho = \int_{\RR^N } f(v)\, dv
$$
and the so-called Darroz\`es-Guiraud  inequality satisfied by operators of the form \eqref{eq:B} (see \cite{DG}):
\begin{align*}
 &\int_{v\cdot n(x)<0} | \mathcal B[\gamma_+ f_\eps] |^2 |v\cdot n(x)| \frac{\d v}{F(v)} \leq \int_{v\cdot n(x)>0} |\gamma_+f_\eps| ^2 |v\cdot n(x)| \frac{\d v}{F(v)} \qquad \mbox{ a.e. } x\in\pa\Omega
\end{align*}
which implies
\begin{align*}
\iint_{\Sigma}  |\gamma f_\eps| ^2 \, v\cdot n(x) \frac{\d S(x)\d v}{F(v)} \geq 0.
\end{align*}

\end{proof}

We then give the following classical result (which can be proved for instance as in \cite{MelletVasseur}):
\begin{proposition}\label{eq:prop}
For all $f_{in} \in L^2_{F^{-1}}(\Omega\times\RR^N)$ there exists a   weak solution of 
\eqref{eq:kin} in the sense of Definition \ref{def:weak} and 
satisfying the energy inequality \eqref{eq:energy}.
\end{proposition}

Inequality \eqref{eq:energy} implies that $f^\eps$ is bounded in 
$L^\infty(0,\infty;L^2_{F^{-1}}(\Omega\times\RR^N))$ and thus converges, up to a subsequence, $\star$-weak to a function $f^0(t,x,v)$.
Note also that 
$$ \int_\Omega |\rho^\eps |^2\, dx \leq \int_\Omega \left|\int_{\RR^N} f^\eps\, dv\right|^2\, dx \leq \int_\Omega \int_{\RR^N} \frac{| f^\eps|^2}{F(v)}\, dv\, dx $$
and so $\rho^\eps$ converges weakly to $\rho(t,x) = \int_{\RR^N}f^0(t,x,v)\, dv$.
Finally, \eqref{eq:energy} also implies that
$$
\lVert f^\eps - \rho^\eps F \lVert_{L^2_{F^{-1}}(\Omega\times\RR^N)} \rightarrow 0 \quad \text{ as } \eps \rightarrow 0. 
$$
and so $f^0-\rho F(v)=0$.
We deduce:

\begin{corollary}
Let $f^\eps$ be weak solution of \eqref{eq:kin} provided by Proposition \ref{eq:prop}. Then, up to a subsequence
$$
f^\eps \rightharpoonup \rho(t,x) F(v) \quad \text{ weakly in } L^\infty(0,+\infty; L^2_{F^{-1}}(\Omega\times\RR^N))
$$
where $\rho(t,x)$ is the weak limit of $\rho^\eps (t,x) = \int_{\RR^N} f^\eps \d v $.
% and, moreover, 
%$$
%\lVert f^\eps - \rho_\eps F \lVert_{L^2_{F^{-1}}(\Omega\times\RR^N)} \rightarrow 0 \quad \text{ as } \eps \rightarrow 0. 
%$$
\end{corollary}

\subsection{Properties of the limiting operators: $D^{2s-1}$ and $\L$}\label{sec:DL}

In this section, we establish some important properties of the operators $D^{2s-1}$ and $\L$. 
First, we need to introduce some classical functional spaces: 
For $\gamma\in(0,1),$ we denote by $C^\gamma(\Omega)$ the set of H\"older continuous functions satisfying
$$ 
\| \vphi \|_{C^\gamma(\Omega)} = \| \vphi\|_{L^\infty(\Omega)} + [\vphi]_{C^\gamma(\Omega)} <\infty
$$
where
$$ [\vphi]_{C^\gamma(\Omega)} = \sup_{x,y\in\Omega\times\Omega} \frac{|\vphi(x)-\vphi(y)|}{|x-y|^\gamma}$$
We also denote by $C^{1,\gamma}(\Omega)$ the set of functions $\vphi$ such that $\vphi\in C^{1}(\Omega)$ and $\na \vphi\in C^\gamma(\Omega)$.

Next, for $s\in(0,1)$ we recall that the fractional Sobolev space $H^s$ is defined by (see \cite{Hitch}):
$$ H^s(\Omega)=\left\{ \vphi \in L^2(\Omega)\, ;\,  \int_\Omega\int_\Omega \frac{(\vphi(x)-\vphi(y))^2}{|x-y|^{N+2s}}\d x\d y <\infty\right\}.$$
It is equipped with the norm:
$$ \| \vphi\|^2_{H^s} = \int_\Omega  |\vphi(x)|^2\d x+ 
 \int_\Omega\int_\Omega \frac{(\vphi(x)-\vphi(y))^2}{|x-y|^{N+2s}}\d x\d y .$$
For $s\in (1,2)$, we also have
$$ H^s(\Omega) = \{\vphi \in H^1(\Omega)\, ;\, \na \phi \in H^{s-1}(\Omega)\}$$
which is equipped with the norm 
$$ \| \vphi\|_{H^s(\Omega)}^2 = \| \vphi\|_{H^1(\Omega)}^2+\| \na \vphi\|^2_{H^{s-1}(\Omega)}.$$
\medskip

Our goal in this section is to prove some results about the operators $D^{2s-1}$ and $\L$ that are used in this paper.
We start by noticing that by (\ref{eq:D0})
$$ |D^{2s-1}[\psi](x)|\leq C \int_\Omega \frac{|\na \psi (y)|}{|y-x|^{N+2s-2}}\, dy.$$
Classical results about Riesz potentials thus implies
\begin{proposition}\label{prop:DbdP}
If $\na \psi\in L^p(\Omega)$ for some $1<p<\frac{N}{2-2s}$, then $D^{2s-1}[\psi]\in L^q(\Omega)$ with
$
q=\frac{N}{N-(2-2s)p}
$
and there exists a constant $C$ such that
$$ \| D^{2s-1}[\psi]\|_{L^q(\Omega)} \leq C\| \na \psi\|_{L^p(\Omega)}.$$
\end{proposition}
\medskip

Next, we note that for $x\in \Omega$ and $\eps<x_N$, we can write (using \eqref{eq:opL2}):
\begin{align*}
\L[\psi](x) &  = \gamma \nu_0^{1-2s}\Gamma(2s)  \int_{\Omega\cap B_\eps(x)}   \frac{y-x}{|y-x|^{N+2s}}  [\na  \psi (y) - \na \psi(x)]   \d y\\
& \quad +  \gamma \nu_0^{1-2s}\Gamma(2s)  \int_{\Omega\setminus B_\eps(x)}   \frac{y-x}{|y-x|^{N+2s}}  \na  \psi (y)  \d y.
\end{align*}
In particular $\L[\psi](x)$ is well defined for all $x\in \Omega$ if $\psi\in C^{1,\gamma}(\Omega)$ for some $\gamma>2s-1$ and $\psi \in W^{1,\infty}(\Omega)$. 
However, when $x$ approaches $\pa\Omega$, the $\eps$ becomes very small and it is difficult to  get a bound on $\L[u]$ up to the boundary.
The next two propositions give necessary and sufficient conditions for such bounds to hold:
\begin{proposition} \label{prop:Linfbd}
Assume that $ \psi \in  C^{1,\gamma}(\Omega)$ for some $\gamma>2s-1$, and $\psi \in W^{1,\infty}(\Omega)$. Then  $\L[\psi] \in L^\infty(\Omega) $ if and only if $x_N^{1-2s} \pa_{x_N}\psi(x)  \in L^\infty(\Omega)$ and so if and only if $ n\cdot \na \psi(x) = 0 $ on $\pa\Omega$.
Furthermore, if $\psi$ satisfies $ n\cdot \na \psi(x) = 0 $  on $\pa\Omega$, then
$$\|  \L[\psi] \|_{ L^\infty(\Omega) } \leq C\| \na \psi \|  _{C^{\gamma}(\Omega)}.$$
\end{proposition}
\begin{proof}
We use Formula \eqref{eq:opL2} for the operator $\L$ and write:
\begin{align}
\L[\psi](x) &  = \gamma \nu_0^{1-2s}\Gamma(2s) \PV \int_{\Omega}   \frac{y-x}{|y-x|^{N+2s}}  \cdot \na \psi (y)    \d y\nonumber \\ 
& = \gamma \nu_0^{1-2s}\Gamma(2s) \int_{\Omega}   \frac{y-x}{|y-x|^{N+2s}}  \cdot [\na \psi (y) -\na \psi(x)]   \d y\nonumber \\
& \quad +\gamma \nu_0^{1-2s}\Gamma(2s) \na \psi(x) \cdot \PV \int_{\Omega}   \frac{y-x}{|y-x|^{N+2s}}     \d y .\label{eq:ggbgh}
\end{align}
The first term in the right hand side is bounded by
$$ C [\na \psi]_{C^\gamma} \int_{\Omega\cap B_1(x)}    \frac{|y-x|^{1+\gamma}}{|y-x|^{N+2s}}  \d y
+  C\| \na \psi\|_{L^\infty} \int_{\Omega\setminus B_1(x)}   \frac{|y-x|}{|y-x|^{N+2s}}   \d y \leq C \| \na \psi\|_{C^\gamma(\Omega)}$$
(we recall that $s\in (1/2,1)$ and $1+\gamma>2s$).
Furthermore, a simple computation shows that
$$
 \PV \int_{\Omega}   \frac{y_i-x_i}{|y-x|^{N+2s}}   \d y
=
\begin{cases}
0 & i=1,\dots, N-1\\
C_{N,s} x_N^{1-2s} & i=N
\end{cases}
$$
so the second term in the right hand side of \eqref{eq:ggbgh} is equal to (up to a constant)
$$
 x_N^{1-2s} \pa_{x_N} \psi(x).
$$
It follows that
$\L[\psi] \in L^\infty(\Omega) $ if and only if $x_N^{1-2s} \pa_{x_N}\psi(x)  \in L^\infty(\Omega)$.

This condition implies that $\pa_{x_N}\psi(x)=0$ on $\pa\Omega$ since $1-2s<0$. 
Furthermore, for such a function, we have 
$$| \pa_{x_N}\psi(x)|  =| \pa_{x_N}\psi(x)-\pa_{x_N}\psi(x',0)|  \leq C [\na \psi]_{C^{\gamma}(\Omega)} |x_N|^{\gamma}$$
and so
$$|x^{1-2s} \pa_{x_N}\psi(x)|  =|x^{1-2s}|| \pa_{x_N}\psi(x)-\pa_{x_N}\psi(x',0)|  \leq C [\na \psi]_{C^{\gamma}(\Omega)} |x_N|^{\gamma+1-2s} \leq C [\na \psi]_{C^{\gamma}(\Omega)} \mbox{ for } x_N\leq 1.$$
Since 
$$|x^{1-2s} \pa_{x_N}\psi(x)|  \leq \| \na \psi\|_{L^\infty} \mbox{ for } x_N\geq 1,$$
we deduce
$$\|  \L[\psi] \|_{ L^\infty(\Omega) } \leq C ([\na \psi]_{C^\gamma} +\| \na \psi\|_{L^\infty}).$$
\end{proof}

We can also prove a similar result in Sobolev spaces:
\begin{proposition} \label{prop:L2bd}
Assume that $\psi \in H^{2s+\beta}(\Omega)$ for some $\beta>0$. Then the following holds:
\item[(i)] If $2s-1<1/2$ (that is $s\in(1/2,3/4)$), then $\L[\psi]\in L^2(\Omega)$ and  
$$\| \L[\psi]\|_{L^2(\Omega)} \leq C \| \psi\|_{H^{2s+\beta}(\Omega)}.$$
\item[(ii)] If $2s-1\geq 1/2$, (that is $s\in[3/4,1)$) then $\L[\psi] \in L^2(\Omega)$ if and only if
\begin{equation}\label{eq:ccdt}
\int_\Omega \left| x_N^{1-2s} \pa_{x_N}\psi\right|^2\d x <\infty
\end{equation}
or, equivalently, if and only if $\pa_{x_N}\psi = 0$ on $\pa\Omega$.
When this condition is satisfied, we then have 
\begin{equation}\label{eq:bdggf}
\| \L[\psi]\|_{L^2(\Omega)} \leq C \| \psi\|_{H^{2s+\beta}(\Omega)}.
\end{equation}
%\begin{equation}\label{eq:L22bd}
%\int_\Omega \left| x_N^{1-2s} \pa_{x_N}\psi\right|^2\d x 
%\leq C \|\psi\|_{H^{2s+\beta}(\Omega)} + \| \mathcal L[\psi] \|_{L^2(\Omega)}.
%\end{equation}
\end{proposition}
Before starting the proof of this proposition, we recall the following Hardy inequality
(see  \cite{LossSloane10,BogdanDyda11}):
\begin{theorem}\label{thm:hardy}
Recall that $\Omega$ is the half space $\{(x_1,\dots , x_N)\, ;\, x_N>0\}$.
Then
for all $s\in(0,1)$ there exists a constant $C$ depending only on $s$ and $N$ such that
for all $f\in C_c(\Omega)$:
\begin{equation}\label{eq:hardy} 
\int_\Omega \frac{|f(x)|^2}{|x_N|^{2s}}\d x \leq C_s \| f \|^2_{\dot{H}^s(\Omega)}
\end{equation}
%(and the inequality hods  for all $f\in C_c(\overline \Omega)$ if $s<1/2$).
\end{theorem}
\begin{remark}\label{rem:hardy}
When $s<1/2$, then $ C_c(\Omega)$ is dense in $H^s(\Omega)$ and so \eqref{eq:hardy} holds for all $f\in H^s(\Omega)$.
When $s> 1/2$, the closure of $ C_c(\Omega)$ in $H^s(\Omega)$ is $H^s_0(\Omega)$, the set of functions in $H^s(\Omega)$ whose trace vanishes at the boundary. In that case \eqref{eq:hardy} holds for all $f\in H^s_0(\Omega)$.
\end{remark}
\begin{proof}[Proof of Proposition \ref{prop:L2bd}]
We write
\begin{align}
\L[\psi](x) &  = \gamma \nu_0^{1-2s}\Gamma(2s)  \int_{\Omega\setminus B_{1}(x)}   \frac{y-x}{|y-x|^{N+2s}}  \cdot \na \psi (y)    \d y\nonumber \\ 
 &  \quad + \gamma \nu_0^{1-2s}\Gamma(2s)  \int_{\Omega\cap B_{1}(x)}   \frac{y-x}{|y-x|^{N+2s}}  \cdot [ \na \psi (y)-\na\psi(x)]    \d y\nonumber \\
 &  \quad + \gamma \nu_0^{1-2s}\Gamma(2s) \na\psi(x) \cdot \PV \int_{\Omega\cap B_{1}(x)}   \frac{y-x}{|y-x|^{N+2s}}    \d y \nonumber \\
 & = I_1(x)+I_2(x)+I_3(x).
  \label{eq:ddf}
 \end{align}
The first term in \eqref{eq:ddf} satisfies (since $s\in(1/2,1)$):
\begin{align*}
\int_\Omega |I_1(x)|^2 \, dx 
& \leq C \int_\Omega \left(  \int_{\Omega\setminus B_{1}(x)}   \frac{1}{|y-x|^{N+2s-1}} | \na \psi (y)  |  \d y\right)^2\, dx\\\
& \leq C \int_\Omega   \int_{\Omega\setminus B_{1}(x)}   \frac{1}{|y-x|^{N+2s-1}}  |\na \psi (y)|^2    \d y \, dx\\
& \leq C \int_\Omega    |\na \psi (y)|^2    \d y
\end{align*} 
 For the second term in \eqref{eq:ddf}, we write
\begin{align*}
\int_\Omega |I_2(x)|^2 \, dx 
 & \leq C \int_\Omega \left(   \int_{\Omega\cap B_{1}(x)}   \frac{| \na \psi (y)-\na\psi(x)|}{|y-x|^{N+2s-1}}      \d y\right)^2\, dx \\
 & \leq C \int_\Omega  \int_{\Omega\cap B_{1}(x)}   \frac{| \na \psi (y)-\na\psi(x)|^2}{|y-x|^{N+2(2s-1+\beta)}}      \d y  \int_{\Omega\cap B_{1}(x)} \frac{1}{|y-x|^{N-2\beta}}\, dy\, dx \\
  & \leq C \int_\Omega  \int_{\Omega}   \frac{| \na \psi (y)-\na\psi(x)|^2}{|y-x|^{N+2(2s-1+\beta)}}      \d y   \, dx \\
  & \leq C  \| \na \psi\|^2_{H^{2s-1+\beta}(\Omega)}\leq C \| \psi\|^2_{H^{2s+\beta}(\Omega)}
 \end{align*}
Finally, for the last term in \eqref{eq:ddf}, we note that  
 $$
\PV \int_{\Omega\cap B_1(x)}  \frac{y_i-x_i}{|y-x|^{N+2s}}\d y 
 = \left\{
 \begin{array}{ll}
 0  & \mbox{ if } i=1,\dots ,N-1\\[3pt]
\displaystyle  \int_{|z|<1, \; z_N>-x_N}  \frac{z_N }{|z|^{N+2s}} \d z & \mbox{ if } i=N
\end{array}
\right.
$$
and we have the following Lemma:
\begin{lemma}\label{lem:h}
The function
$$ h(x_N) = \PV  \int_{|z|<1, \; z_N>-x_N} \frac{z_N }{|z|^{N+2s}} \d z , \qquad x_N>0$$
satisfies $h(x_N) =0 $ if $x_N\geq 1$ and 
$$C_1|x_N|^{1-2s} \leq h(x_N)\leq C_2 |x_N|^{1-2s} \qquad \mbox{ for } 0<x_N<1/2
$$
and
$$0\leq h(x_N)\leq C_3\qquad \mbox{ for } 1/2<x_N<1
$$
\end{lemma}
Postponing the proof for now, we note that this lemma implies that
$$
 \int_\Omega  |x_N|^{2(1-2s)} 1_{\{x_N\leq 1/2\} } \left|  \pa_{x_N} \psi\right|^2\d x \leq \int_\Omega |I_3(x)|^2 \, dx \leq  \int_\Omega  |x_N|^{2(1-2s)}   1_{\{x_N\leq 1\} } \left|  \pa_{x_N} \psi\right|^2 \d x
$$

Since $I_1$ and $I_2$ in \eqref{eq:ddf} are in $L^2(\Omega)$ when $\psi \in H^{2s+\beta}(\Omega)$, we  deduce that $\L[\psi]$ belongs to $L^2(\Omega)$ if and only if $I_3$ is in $L^2(\Omega) $ as well, which is then equivalent to  \eqref{eq:ccdt}.
%$$ 
%\int_\Omega \frac{\left|  \pa_{x_N}\psi\right|^2}{|x_N|^{2(2s-1)}}\d x <\infty.$$

We can now complete the proof of the Proposition:
\medskip

\noindent (i) When $2s-1<1/2$ (that is $s<3/4$), Hardy's inequality (see Remark \ref{rem:hardy}) implies 
$$ 
\int_\Omega \frac{\left|  \pa_{x_N}\psi\right|^2}{|x_N|^{2(2s-1)}}\d x \leq C \| \na \psi\|^2_{H^{2s-1}(\Omega)}$$
and so $\L[\psi]\in L^2(\Omega)$ without further conditions  and the bound on $\| \L\psi]\|_{L^2(\Omega)}$ follows from the bounds on $I_1$ and $I_2$ above. 

\medskip

\noindent (ii) When $2s-1\geq 1/2$, then we proved above that $\L[\psi]$ belongs to $L^2(\Omega)$ if and only if  \eqref{eq:ccdt} holds.

Furthermore, 
since  $\na \psi \in H^{2s-1+\beta}(\Omega)$ with $2s-1+\beta>1/2$ we see that $\pa_{x_N} \psi$ has a well defined trace in $L^2(\pa\Omega)$ and \eqref{eq:ccdt} implies that this trace must vanish since $2(2s-1)\geq 1$. Conversely, if $\pa_{x_N} \psi=0$ on $\pa\Omega$, then $\pa _{x_N} \psi$ belongs to $H^{2s-1+\beta}_0(\Omega)$, the closure of $C^\infty_0(\Omega)$. 
Hardy inequality (see Remark \eqref{rem:hardy}) thus implies
\begin{equation}\label{eq:blff}
\int_\Omega \frac{\left|  \pa_{x_N}\psi\right|^2}{|x_N|^{2(2s-1+\beta)}}\d x \leq C \| \na \psi\|^2_{H^{2s-1+\beta}(\Omega)}
\end{equation}
and so \eqref{eq:ccdt}  holds.

We have thus shown that \eqref{eq:ccdt} was equivalent to the condition that $\pa_{x_N} \psi=0$ on $\pa\Omega$.
When this condition holds, then the inequality above gives
$$
\int_\Omega \frac{\left|  \pa_{x_N}\psi\right|^2}{|x_N|^{2(2s-1+\beta)}}\d x \leq C \|  \psi\|^2_{H^{2s+\beta}(\Omega)}$$
and \eqref{eq:bdggf} follows.
\end{proof}
\begin{proof} [Proof of Lemma \ref{lem:h}]

Using symmetry properties, we write:
\begin{align*}
h(x_N)&= \int_{z_N>x_N; \; {z'}^2+z_N^2<1}\frac{z_N}{(z_N^2+{z'}^2)^{\frac{N+2s}{2}}}dz\\
\end{align*}
Thus, by proceeding to the change of variable $y'=\frac{z'}{z_N}$ we find
\begin{align*}
h(x_N)&= \int_{x_N<z_N<1} \frac{z_N}{z_N^{N+2s}}z_N^{N-1}\left(\int_{{y'}^2+1<\frac{1}{{z_N}^2}}\frac{1}{(1+{y'}^2)^\frac{N+2s}{2}}dy'\right)dz_N\\
&\leq  \int_{x_N<z_N<1}z_N^{-2s}dz_N\left(\int_{\RR^{N-1}}\frac{1}{(1+{y'}^2)^\frac{N+2s}{2}}dy'\right)\\
&\leq \frac{C}{2s-1}({x_N}^{1-2s}-1),
\end{align*}
which gives the desired upper bounds for $0<x_N <1$.

On another hand, we clearly have $h(x_N)\geq 0$ and we can also write
\begin{align*}
h(x_N)&\geq  \int_{x_N<z_N<\frac{3}{4}}z_N^{-2s}dz_N\left(\int_{{y'}^2+1<\frac{16}{9}}\frac{1}{(1+{y'}^2)^\frac{N+2s}{2}}dy'\right)\\
&\geq  \frac{C}{2s-1}({x_N}^{1-2s}-(\frac{3}{4})^{1-2s})
%& \geq C|x_N|^{1-2s}
\end{align*}
which gives the lower bound when 
$0<x_N<\frac{1}{2}$. 
\end{proof}

We deduce the following Corollary which is useful in the proof of our main theorem:
\begin{corollary}\label{cor:15}
If $\psi \in H^{2s+\beta}(\Omega)$ for some $\beta>0$ and $\L[\psi]\in L^2(\Omega)$, then
\begin{equation}\label{eq:15}
 \delta^{-2(2s-1)}\int_\Omega|  \pa_{x_N} \psi(t,x)|^2 1_{x_N\leq \delta} \d x \leq  \|  \psi\|^2_{H^{2s +\beta}(\Omega)} \delta^{2\beta'}
\end{equation}
for some $\beta'>0$.
\end{corollary}
\begin{proof}
When $2s-1<1/2$ (that is $s<3/4$), we can take $\beta'<\beta$ such that $2s-1+\beta'<1/2$
and Hardy's inequality implies 
$$ 
\int_\Omega |x_N|^{-2(2s-1+\beta')}  \left|  \pa_{x_N}\psi\right|^2 \d x \leq C \| \na \psi\|^2_{H^{2s-1+\beta'}(\Omega)}\leq C  \|  \psi\|^2_{H^{2s +\beta}(\Omega)} .$$
Since $ \delta^{-2(2s-1+\beta')} \leq |x_N|^{-2(2s-1+\beta')}$ when $x_N\leq \delta$, we deduce
$$ \delta^{-2(2s-1+\beta')} \int_\Omega |  \pa_{x_N} \psi(t,x)|^2 1_{x_N\leq \delta} \d x \leq C   \|  \psi\|^2_{H^{2s +\beta}(\Omega)} $$
%and so
%$$ \delta^{-2(2s-1)} \int_\Omega|  \pa_{x_N} \psi(t,x)|^2 1_{x_N\leq \delta} \d x \leq C   \|  \psi\|^2_{H^{2s +\beta}(\Omega)} \delta^{2\beta'} $$
and \eqref{eq:15} follows.

When $2s-1\geq 1/2$, \eqref{eq:15} (with $\beta'=\beta$) follows by a similar computation using \eqref{eq:blff}.
% we proceed as in the proof of Proposition \ref{prop:L2bd} above to show that the condition $\L[\psi]\in L^2(\Omega)$ implies $\pa _{x_N} \psi =0$ on $\pa\Omega$ and that Hardy inequality   in turn implies
%$$
%\int_\Omega  |x_N|^{-2(2s-1+\beta)}  \left|  \pa_{x_N}\psi\right|^2 \d x \leq C \| \na \psi\|^2_{H^{2s-1+\beta}(\Omega)}$$
%and \eqref{eq:15} follows as before (with $\beta'=\beta$).
\end{proof}

Finally, we prove the following integration by part formula for $\div D^{2s-1}$ (that we will prove to be $\L$):
\begin{proposition}\label{prop:ipp}
%The operator $\mathcal L$ defined by
%$$\mathcal L [\vphi]=- \div D^{2s-1}[\vphi] $$
%satisfies 
Let $\psi$ and $\vphi$ be functions  in $H^{2s+\beta}(\Omega)$ for some $\beta>0$ such that  $\L[\psi]$ and $\L[\vphi] \in L^2(\Omega) $. Then
the following integration by parts formula holds:
\begin{equation}\label{eq:ipp}
\int_{\Omega} \div D^{2s-1}[\vphi] \psi \d x - \int_{\Omega} \vphi \, \div D^{2s-1} [\psi]\d x  = \int_{\pa \Omega} \left[\psi D^{2s-1}[\vphi]\cdot n -
 \vphi D^{2s-1}[\psi]\cdot n \right]\d S(x).
\end{equation}
\end{proposition}
Note that we can also prove that  this formula holds when $\vphi$ and $\psi$ are in $C^{1,\gamma}(\Omega)$ for some $ \gamma>2s-1$ and satisfies the Neumann condition $\pa_{x_N}\vphi = \pa_{x_N}\psi =0 $ on $\pa\Omega$
(see Proposition \ref{prop:Linfbd}).
%The formula \eqref{eq:ipp},  generalizes the usual formula
%$$ \int_{\Omega} \Delta \vphi\,  \psi\d x - \int_{\Omega} \vphi \Delta \psi \d x  = \int_{\pa \Omega}\psi \na \vphi \cdot n -
% \vphi \na \psi\cdot n \d S(x).
% $$ 
%justifies the fact that \eqref{eq:weakas} is the weak formulation of \eqref{eq:as}
%in the same way that
%$$
%\begin{cases}
%& \underset{\RR^+\times\Omega} {\iint} \rho \Big(  \pa_t \psi +\Delta \psi   \Big) \d t\d x + \underset{\Omega}{\int} \rho_{in} (x) \psi(0,x) \d x =0 \\
%& \mbox{for all test function $\psi$ satisfying $\na \psi \cdot n=0$ on $\pa\Omega$.} 
%\end{cases}
%$$
%is a weak formulation for the usual heat equation with Neumann boundary conditions.

\begin{proof}[Proof of Proposition \ref{prop:ipp}]
Integrating by parts, we find:
\begin{equation} \label{eq:ipp1}
\int_{\Omega}  \div D^{2s-1}[\vphi]  \psi\d x = - \int_{\Omega} D^{2s-1}[\vphi]\cdot \na \psi \d x 
 + \int_{\pa \Omega} \psi D^{2s-1}[\vphi]\cdot n \d S(x).
\end{equation}
So, formula \eqref{eq:ipp} follows from the following equality:
$$
 \int_{\Omega} D^{2s-1}[\vphi]\cdot \na \psi \d x =  \int_{\Omega} D^{2s-1}[\psi]\cdot \na \vphi \d x.
$$
This equality is easily proved using the formula \eqref{eq:D0}  for the operator $D^{2s-1}$ since it  gives the following symmetric expression:
\begin{equation}\label{eq:ipp3} 
 \int_{\Omega} D^{2s-1}[\vphi]\cdot \na \psi \d x =  \gamma \nu_0^{1-2s} \Gamma(2s-1)
 \int_{\Omega}   \int_\Omega (y-x)\cdot \na \vphi(y)  \frac{y-x}{|y-x|^{N+2s}}\cdot \na \psi(x) \, dy  \d x. 
 \end{equation}

\end{proof}

%\begin{lemma}\label{lem:ipp}
%For all functions $\vphi$, $\psi$ as in Proposition \ref{prop:ipp}, we have:
%\begin{equation}\label{eq:ipp2}
% \int_{\Omega} D^{2s-1}[\vphi]\cdot \na \psi \d x =  \int_{\Omega} D^{2s-1}[\psi]\cdot \na \vphi \d x.
%\end{equation}
%\end{lemma}
%Postponing the proof of this lemma, we see that \eqref{eq:ipp1} and \eqref{eq:ipp2} imply
%\begin{align*}
%\int_{\Omega} \div D^{2s-1}[\vphi] \psi\d x 
%& =  - \int_{\Omega} D^{2s-1}[\vphi]\cdot \na \psi \d x  + \int_{\pa \Omega} \psi D^{2s-1}[\vphi]\cdot n \d S(x) \\
%& =  - \int_{\Omega} D^{2s-1}[\psi]\cdot \na \vphi \d x  + \int_{\pa \Omega} \psi D^{2s-1}[\vphi]\cdot n \d S(x) \\
%& = \int_{\Omega}\div D^{2s-1} [\psi]  \vphi \d x - \int_{\pa \Omega} \vphi D^{2s-1}[\psi]\cdot n \d S(x)   
%+ \int_{\pa \Omega} \psi D^{2s-1}[\vphi]\cdot n \d S(x) .
%\end{align*}
%

\section{Proof of Theorem \ref{thm:main}} \label{sec:rig}
In this section, we rigorously prove  the limit presented in the previous section in a particular case: we assume that $\Omega=\RR^N_+$ is the upper-half space and that the collision cross-section is constant, so that
\begin{equation}\label{eq:Q0} 
Q(f)(v)=\nu_0 (\rho F(v)-f(v)), \qquad \rho = \int_{\RR^N} f(v)\d v.
\end{equation}
Furthermore, we assume that $F$ satisfies
\begin{equation}\label{eq:F}
\begin{cases}
\displaystyle F(v) \in L^\infty, \qquad \int_{\RR^N} F(v)\d v =1, \qquad F(v)=F(-v) \\[5pt]
\left| F(v)-\frac{\gamma}{|v|^{N+2s}} \right|\leq \frac{C}{|v|^{N+4s}} \qquad\mbox{ for all } |v|\geq 1.
\end{cases}
\end{equation}
These assumptions on the equilibrium $F$ are motivated by the equilibrium of the fractional Fokker-Planck operator studied in \cite{CesbronMelletTrivisa12}. 
\medskip

The basic idea of the proof is to rigorously pass to the limit in \eqref{eq:weakeps} (note that when $Q$ is given by \eqref{eq:Q0}, the last term in \eqref{eq:weakeps}  vanishes).
To do this, we would like to solve \eqref{eq:Phieps}-\eqref{eq:Phibceps} explicitly, which is difficult because of the boundary condition. 
Instead, we will construct solutions of the following equation
\begin{equation}\label{eq:Phiepsff}
\begin{cases}
\nu_0 \phi^\eps - \eps v \cdot \na_x\phi^\eps  = \nu_0 \psi(x) \qquad & \mbox{ in } \Omega\times\RR^N\\
\gamma_+ \phi^\eps = \psi(x) & \mbox{ on } \Sigma_+.
\end{cases}
\end{equation}
The function $\phi^\eps$ then satisfies the boundary condition \eqref{eq:Phibceps} if and only if (see Lemma \ref{lem:test} below)
\begin{equation}\label{eq:ope}
\int_{\RR^N}   v F(v) \big[\phi^\eps(t,x,v)-\psi(t,x)\big] \, \d v  \cdot n(x) =0\; , \qquad \mbox{ for all } x\in \pa\Omega,
\end{equation}
which leads us to introduce the following operator
\begin{equation}\label{eq:defDDeps} D^{2s-1}_\eps [\psi](x):= \eps^{1-2s}\int_{\RR^N}   v F(v) \big[\phi^\eps(t,x,v)-\psi(t,x)\big] \, \d v , \qquad x\in\Omega\end{equation}
(note that this operator coincides with the operator $\tilde  D^{2s-1}_\eps$ of the previous section when $\psi$ satisfies \eqref{eq:ope}, but is otherwise different).

However, condition \eqref{eq:ope} depends on $\eps$, so this approach would require us to consider a sequence of test function $\psi^\eps$ satisfying \eqref{eq:ope} and converging to $\psi$ when $\eps\to0$. Since the existence of such a sequence is not clear, we will instead fix a function $\psi$ such that
 $$ \lim_{\eps\to 0} D^{2s-1}_\eps [\psi]\cdot n=0 $$
and show that the corresponding function $\phi^\eps$, solution of \eqref{eq:Phiepsff} can be approximated by a function satisfying the boundary conditions  \eqref{eq:Phibceps}. 
%We will instead show that for functions $\psi$ satisfying the Neumann condition \eqref{eq:psibc3}, 
%we can construct approximated solutions of  \eqref{eq:Phieps}-\eqref{eq:Phibceps} which 
%This will allow us to carry out the formal argument outlined in the previous section with a slight different operator $D_\eps^{2s-1}$ (which is equivalent to the operator $\tilde D_\eps^{2s-1}$ introduced above provided $\psi$ satisfies the boundary condition \eqref{eq:psibc3}).
This last approximation is the main reason why we only prove our result when $\Omega$ is the upper-half space since the construction is significantly simpler in that case.

In the next section, we introduce an extension of $\psi$ to $\RR^N\times\RR^N$ which will lead to an explicit formula for the solution of  \eqref{eq:Phiepsff}.
%the  construction of the approximate test functions (see Section \ref{sec:approxtest}). 
We will then proceed with the  proof of our main theorem with, in particular, the proof of the convergence of the operators $D^{2s-1}_\eps$ and $\L^\eps$.

\subsection{Construction of the test functions}   \label{sec:approxtest}
Our first task is to explicitly solve equation \eqref{eq:Phiepsff} for a given a test function $\psi(t,x)$ defined in $[0,\infty)\times \overline \Omega$.
To do that, we first define an extension $ 
\wpsi (t,x,v) $ of $\psi$ to $[0,\infty)\times\RR^N\times\RR^N$ 
by setting
\begin{equation}\label{eq:defwpsi00}
\wpsi (t,x,v) = \psi(t,x) \mbox{ for all $x\in\overline \Omega$ and all $v\in \RR^N$}
\end{equation}
and assuming that $\wpsi(x,v,t)$ solves
\begin{equation}\label{eq:defwpsi01}
\left\{
\begin{aligned}
 v\cdot \na _x \wpsi(t,x,v) = 0 & \qquad \mbox{ in } (\RR^N\setminus \Omega) \times\RR^N \\
 \wpsi(t,x,v) = \psi(t,x) &\qquad \mbox{ on } \Sigma_+.
\end{aligned}
\right.
\end{equation}
This equation states that  for fixed $(t,v)$, the function $x\mapsto \wpsi(t,x,v) $ is constant along the characteristic lines $\tau \mapsto x+\tau v$ outside the set $\Omega$.
Since not all characteristic lines will intersect $\Sigma_+$, this does not define $\wpsi(x,v)$ uniquely everywhere. However, we will see that the ambiguous points do not play any role in the sequel, so we can set $\wpsi(x,v)$ to be zero there.
We note the following obvious but important facts about this extension
\begin{enumerate}
\item  For any $x\in \pa\Omega$, and any $v\in\RR^N$ such that $v\cdot n(x)>0$, we have
\begin{equation}\label{eq:wpsichar}
\wpsi(t,x+\tau v,v) = \psi(t,x) \qquad \forall \tau\geq0 .
\end{equation}
\item If $x\in \overline \Omega$, then $ \wpsi(t,x,\tau v) = \psi(t,x) = \wpsi(t,x,v)$ for any $v\in\RR^N$ and any $\tau \in\RR$.
If $x\in\RR^N\setminus \Omega$,  
since the boundary condition in \eqref{eq:defwpsi01} does not depend on $v$, the function $\tau \mapsto \wpsi(t,x,\tau v)$ is constant for any $v\in \RR^N$ (check, for instance, that $\wpsi(t,x,\tau v)$ is also a solution of \eqref{eq:defwpsi01}). We deduce
$$
 \wpsi(t,x,\tau v) =  \wpsi(t,x,v) \qquad \forall \tau >0, \quad \forall (t,x,v)\in (0,\infty)\times \RR^N\times \RR^N.
$$
\item This construction is useful for any convex set $\Omega$, but when $\Omega$ is the upper-half plane, we can get the following explicit formula:
\begin{equation}\label{eq:fhhr} \wpsi(t,x,v) = 
\begin{cases}
\psi(t,x) & \mbox{ if } x_N>0 \\
\psi\left(t,x'-\frac{x_N}{v_N} v',0\right) & \mbox{ if } x_N<0, \quad v_N<0
\end{cases} 
\end{equation}
As noted above, this does not define $ \wpsi(t,x,v)$ for $x\notin\Omega$ and $v_N>0$, but these value do not play any role in what follows.
We also have
\begin{equation}\label{eq:sim}
\wpsi(y,y-x)=\left\{ \begin{array}{ll} \psi(x) & \mbox{ if } y_N<0\\ \psi(y) & \mbox{ if } y_N >0 \end{array}\right. 
\qquad \mbox{ for all }x\in\pa\Omega.
\end{equation}

\item Even of $\psi$ is smooth, we do not expect $\wpsi$ to be regular near $\pa\Omega$.
For example, the function $x\mapsto \wpsi(x,v)$ is $C^0$ (but typically not $C^1$) at a point $x_0\in\pa\Omega$ only if $n(x_0)\cdot v>0$.
Nevertheless, it is difficult to show that if $x\mapsto \psi(x)$ is in $C^\alpha(\Omega)$ for some $\alpha\in(0,1)$, then we have 
\begin{equation}\label{eq:psilip}
|\wpsi(y,y-x) - \psi(x) |\leq C|y-x|^\alpha \qquad\forall x\in \Omega, \; y\in\RR^N.
\end{equation}
\end{enumerate}

We now have the following Lemma:
\begin{lemma}\label{lem:test}
For any test function $\psi(t,x)$ defined in $[0,\infty)\times \overline \Omega$, the function
\begin{equation}\label{eq:phiepsdef} \phi^\eps(t,x,v) =  \int_0^\infty e^{-\nu_0 z}\nu_0 \wpsi (t,x+\eps v z,v)\d z
\end{equation}
solves \eqref{eq:Phiepsff}.
Furthermore, $\phi^\eps$ satisfies the boundary condition \eqref{eq:Phibceps} if and only if $\psi$ is such that
\begin{equation}\label{eq:psibc2}
\int_{\RR^N}   v F(v) \big[\phi^\eps(t,x,v)-\psi(t,x)\big] \, \d v  \cdot n(x) =0\; , \qquad \mbox{ for all } x\in\pa\Omega.
\end{equation}
\end{lemma}

With $\phi^\eps$ is given by \eqref{eq:phiepsdef}, the operator defined in (\ref{eq:defDDeps}) becomes:
\begin{equation}\label{eq:DDss}
D^{2s-1}_\eps[\psi](x) = \eps^{1-2s} \int_{\RR^N}   \int_0^\infty e^{-\nu_0 z} \nu_0   v F(v) [\wpsi (x+\eps v z,v) -\psi(x)]\, dz\d v.
\end{equation}
We also introduce the limiting operator 
\begin{equation}
\label{eq:D00} 
 D^{2s-1} [\psi] (x):= \gamma \nu_0^{1-2s}\Gamma(2s)  \int_{\RR^N}  \big[\wpsi(y,y-x)-\psi(x)\big] \frac{y-x}{|y-x|^{N+2s}} \,dy .
\end{equation}
We will prove in Proposition \ref{prop:d2s} below that $D^{2s-1}_\eps [\psi]$ converges to $D^{2s-1} [\psi]$ and that \eqref{eq:D00} is equivalent to the formula  \eqref{eq:D0} given in the introduction (and which does not involve the extension $\wpsi$).

\begin{proof}[Proof of Lemma \ref{lem:test}]
Using \eqref{eq:defwpsi00}, we easily check that $\phi^\eps$ satisfies \eqref{eq:Phieps},
and using \eqref{eq:wpsichar} (this is where the definition of the extension $\wpsi$ is crucial), we see that for all $x\in\pa\Omega$ and $v$ such that  $v\cdot n(x)>0$, we have:
$$ \gamma_+ \phi^\eps  (t,x,v)  =  \int_0^\infty e^{-\nu_0 z} \nu_0 \wpsi (t,x+\eps v z,v)\d z  = \int_0^\infty e^{-\nu_0 z} \nu_0 \psi (t,x)\d z  = \psi(t,x).$$
which is the boundary condition in \eqref{eq:Phiepsff}.
Next, we note that the boundary condition \eqref{eq:Phibceps} is satisfied if and only if
$$ 
 \gamma_+ \phi^\eps =\psi(x,t) = \int_{w\cdot n(x)<0 }\alpha_0 F(w) \phi^\eps(x,v) |w\cdot n(x)|  \d w \quad\mbox{ for all } x\in\pa\Omega.
$$
%$$\int_{w\cdot n(x)<0}\int_0^\infty e^{-z}\alpha_0 F(w) \wpsi(t,x+\eps  w z ,w) |w\cdot n(x)| \d z \d w$$
Using the normalization condition $\int_{w\cdot n(x)<0 }\alpha_0 F(w)|w\cdot n(x)|  \d w=1$, we can rewrite this condition as
\begin{equation}\label{eq:psibc}
\int_{w\cdot n(x)<0}   w F(w) \big[\phi^\eps(t,x,v)-\psi(t,x)\big] \d w\cdot n(x) =0\; , \qquad \mbox{ for all } x\in\pa\Omega.
%\int_{w\cdot n(x)<0}  \int_0^\infty  e^{-z} w F(w) \big[\wpsi(t,x+\eps w z,w)-\psi(t,x)\big] \d z\d w\cdot n(x) =0\; , \qquad \mbox{ for all } x\in\pa\Omega.
\end{equation}
Finally, since $\phi^\eps  (t,x,v) = \psi(t,x)$ when $x\in\pa\Omega$ and  $v\cdot n(x)> 0$, 
 we can extend the integral over all $w\in\RR^N$ and  write this condition as \eqref{eq:psibc2}.
\end{proof}
\medskip

Since equation \eqref{eq:psibc2} depends on $\eps$ and we want to work with a fixed $\psi$, we will assume that $\psi$ satisfies the limiting Neumann boundary condition  
 \eqref{eq:psineumann}: 
$$
D^{2s-1} [\psi] (t,x) \cdot n(x) = 0 \qquad \text{ for  all } x\in \dO.
$$
% 
%We cannot expect a function $\psi$ to satisfy  \eqref{eq:psibc2} for all $\eps>0$, but  if $\psi$ satisfies the Neumann boundary condition
% \eqref{eq:psibc3}, then the convergence of $D^{2s-1}_\eps$ to $D^{2s-1}$ suggests that $\phi^\eps$ almost satisfies \eqref{eq:psibc2} for small  $\eps$.
While this implies that $\psi$ almost satisfies \eqref{eq:psibc2} for small $\eps$, it is not enough since we need $\phi^\eps$  to satisfy  \eqref{eq:Phibceps} in order to take it as a test function in  \eqref{eq:weak0}.
We will thus now approximate $\phi^\eps$ by a new function $\phi_0^\eps$ which is an exact solution of the boundary condition \eqref{eq:psibc2} (but an approximated solution of the transport equation\eqref{eq:Phieps}).

This construction is simpler when $\Omega$ is the upper-half space so we restrict ourselves to this case from now on. 
In particular, using \eqref{eq:sim} and \eqref{eq:D00}, the Neumann boundary condition \eqref{eq:psineumann} can then  be written as
\begin{align*} 
\int_{\RR^N_+}  \big[\psi(y)-\psi(x)\big] \frac{(y-x)\cdot n}{|y-x|^{N+2s}} \,dy=0   \qquad \text{ for  all } x=(x',0)\in \dO
\end{align*}
or equivalently
\begin{align} \label{eq:limitBCpsi}
\int_{\RR^N_+}  \big[\psi(x+y)-\psi(x)\big] \frac{y\cdot n}{|y|^{N+2s}} \,dy=0   \qquad \text{ for  all } x=(x',0)\in \dO
\end{align}
where the outward normal vector $n$ is given by $n=(0,\dots,0,-1)$.

We now introduce the following approximation of $\phi^\eps$:
\begin{align}\label{eq:newphi}
\phi_0^\eps (t,x,v) =  \phi^\eps (t,x,v) + T^\eps(t,x) \chi(v).
\end{align}
where $\chi$ is a smooth function compactly supported in $\RR^N_+$, satisfying
$$ \chi(v)=0 \mbox{ if } v_N <0, \quad \mathcal B^*[\chi](v)=\int_{w\cdot n< 0}  \alpha_0 F(w) \chi(w) |w\cdot n| \d w  = 1.$$
Using \eqref{eq:sim}, we see that the function $\phi^\eps (t,x,v)$ satisfies the boundary condition \eqref{eq:Phibceps} if and only if
\begin{align*}
\psi(t,x) & =\mathcal B^*[ \phi^\eps (t,x,\cdot)+T^\eps(t,x) \chi(v)] \\
& = \mathcal B^*[ \phi^\eps (t,x,\cdot) ] + T^\eps(t,x)  \mathcal B^*[\chi(v)]\\
& =\mathcal B^*[\phi^\eps (t,x,\cdot) ] + T^\eps(t,x)  \qquad \text{ for  all } x=(x',0)\in \dO.
\end{align*}
So $T^\eps(t,x)$ must satisfy
\begin{align*}
T^\eps (t,x) & = - \int_{w\cdot n<0}\alpha_0 F(w) (\phi^\eps (t,x,v) -\psi(x))  |w\cdot n | \d w \nonumber \\
& = 
\eps^{2s-1} \alpha_0 D^{2s-1}_\eps [\psi] (t,x) \cdot n 
 \qquad \forall x=(x',0)\in\pa\Omega.
 %\label{eq:defT0}
\end{align*}
%where we denoted $\Ox=(x',0)$.
We thus define
$$\oTe(t,x') := D^{2s-1}_\eps [\psi] (t,x',0) \cdot n \quad \mbox{ for  } x'\in\RR^{N-1}$$
and set
\begin{equation}\label{eq:defT} 
T^\eps(t,x) =\eps^{2s-1} \alpha_0 \oTe(t,x') e^{-{x_N}^2} \quad \mbox{ for  } x\in\Omega.
\end{equation}
%\begin{align*}
%T^\eps (t,x',0) = \eps^{2s-1} D^{2s-1}_\eps [\psi] (t,x) \cdot n
%\end{align*}
Note that the addition of the corrector $ T^\eps(t,x) \chi(v)$ in \eqref{eq:newphi} guarantees that the function $\phi_0^\eps$ satisfies the boundary condition  \eqref{eq:Phibceps}, but it no longer satisfies
the transport-like equation \eqref{eq:Phieps}. 
However, we will show in Proposition \ref{prop:corrector} that $T^\eps$ goes to zero as $\eps\to0$.
The following remark ensures that it is still an admissible test function in the sense of Definition \ref{def:weak}:
\begin{remark}
The definition of $\phi^\eps$, \eqref{eq:phiepsdef}, implies
\begin{align*}
\underset{\Omega \times\RR^N}{\iint} |\phi^\eps(t,x,v)|^2 F(v)\d x\d v
& \leq \underset{\Omega \times\RR^N}{\iint} \int_0^\infty e^{-\nu_0 z} \nu_0 |\wpsi^\eps(t,x+\eps z v,v)|^2 F(v)\d z\d x\d v
\end{align*}
Using \eqref{eq:L2bdt} (proved below), we deduce
\begin{align*}
\underset{\Omega \times\RR^N}{\iint} |\phi^\eps(t,x,v)|^2 F(v)\d x\d v
& \leq \int_{\RR^N} \int_0^\infty e^{-\nu_0 z} \nu_0 C(\psi) (1+|v_N|) F(v)\d v\\
& \leq  C(\psi)
\end{align*}
for some constant $C(\psi)$ depending on $\|\psi\|_{L^2(\RR_+\times\RR^N_+)}$ and $\|\psi|_{\pa\Omega}\|_{L^2(\RR_+\times\RR^{N-1})}$.
A similar bound holds for $\pa_t\phi^\eps$ since $t$ is a parameter in the definition of $\phi^\eps$.
Furthermore, Equation \eqref{eq:Phieps} then implies that
$$ \underset{\Omega \times\RR^N}{\iint} |\eps v\cdot \na_x \phi^\eps(t,x,v)|^2 F(v)\d x\d v <\infty$$
From there, it is easy to check that  we can indeed take $\phi_0^\eps$ as a test function in \eqref{eq:weak0}.
\end{remark}

We can now proceed as in Section \ref{sec:formal}: We take  the  function $\phi_0^\eps$ constructed above as test function in the weak formulation  \eqref{eq:weak0} of \eqref{eq:kin}.
We obtain:
\begin{align}
&\underset{\RR_+\times\Omega \times\RR^N}{\iiint} f^\eps  \pa_t \phi^\eps   \d t \d x \d v  + \underset{\RR_+\times\Omega}{\iint} \rho^\eps \mathcal L^\eps [\psi] \d t \d x \d v \label{eq:wftildepsi} \\
&+ \underset{\RR_+\times\Omega\times\RR^N}{\iiint} f^{\eps} \Big( \pa_t T^\eps(t,x) \chi(v) + \eps^{1-2s} v \cdot \na_x  T^\eps(t,x) \chi(v) - \eps^{-2s} Q^* \big[ \chi(v) \big] T^\eps(t,x)  \Big) \d t \d x \d v  \label{eq:wfCorrection} \\
&+ \underset{\Omega\times\RR^N}{\iint} f_{in} (x,v) \phi^\eps(0,x,v) \d x \d v =0  \label{eq:wfinitial}
\end{align}
with
\begin{align} 
\L^\eps[\psi](x) &  := \eps^{-2s}\int_{\RR^N}\nu_0 F(v)   [  \phi^\eps(x,v) -  \psi(x)]  \d v \label{eq:defL1'} \\
&  = \eps^{-2s}\int_{\RR^N}  F(v) \eps v \cdot \na_x \phi^\eps(x,v) \d v  \label{eq:defL3} \\
& = \div_x\left( \eps^{1-2s} \int_{\RR^N}   v F(v) \phi^\eps(x,v) \d v \right)\label{eq:defL2'}\\
& = \div_x\left( \eps^{1-2s} \int_{\RR^N}   v F(v) [\phi^\eps(x,v) -\psi(x)]\d v \right)\nonumber\\
& = \div_xD^{2s-1}_\eps[\psi](x).\label{eq:defL3'}
\end{align}

This equation differs from \eqref{eq:weakeps} in two important ways: 
Because $Q$ is given by \eqref{eq:Q0}, the last term in \eqref{eq:weakeps} does not appear in \eqref{eq:wftildepsi} (since $K(g^\eps)  = \int_{\RR^N} g^\eps(v)\, dv=0$).
On the other hand, the construction of $\phi^\eps_0$ has given rise to the additional terms \eqref{eq:wfCorrection} in the second line and we will need to show that these terms vanishe in the limit (this is where we need $\psi$ to satisfy the Neumann boundary condition).
The rest of the proof consist in passing to the limit in this equation.

%Our goal is thus to rigorously justify the limit in the terms \eqref{eq:wftildepsi}, which will be the subject of Section \ref{sec:derivation}, and to show in particular  that the terms \eqref{eq:wfCorrection} vanish in the limit. 
\medskip

We conclude this subsection with the following simple lemma which will be useful several times throughout the paper:
\begin{lemma}\label{lem:L2bd}
For all $\psi\in H^s(\RR_+^N)$, $s>1/2$, and $v\in\RR^N$  we have:
\begin{equation}\label{eq:L2bd}
\int_{\RR^N_+} |\wpsi(x+v,v)|^2\d x \leq \left\{
\begin{array}{ll}
\int_{\RR^N_+} |\psi(x)|^2\d x & \mbox{ if } v_N>0 \\
\int_{\RR^N_+} |\psi(x)|^2\d x + |v_N| \int_{\RR^{N-1}} |\psi(x',0)|^2\d x' & \mbox{ if } v_N<0 \\
\end{array}
\right.
\end{equation}
\end{lemma}
In the sequel, we will repeatedly use the inequality (which follows from \eqref{eq:L2bd}):
\begin{equation}\label{eq:L2bdt}
\int_{\RR_+} \int_{\RR^N_+} |\wpsi(t,x+v,v)|^2\d x\d t \leq C(1+|v_N|)\|\psi\|_{H^s(\RR_+\times\RR^N_+)}.
\end{equation}
\begin{proof}[Proof of Lemma \ref{lem:L2bd}]
If $v_N>0$, then $x+v\in\Omega$ for all $x\in \Omega$ and so $\wpsi(x+v,v)=\psi(x+v)$. We deduce
$$ \int_{\RR^N_+} |\wpsi(x+v,v)|^2\d x = \int _{\RR^N_+}| \psi (x+v )|^2\d x = \int _{\RR^N_+, x_N\geq v_N}| \psi (x )|^2\d x 
$$
and the first inequality follows.

If $v_N<0$, then we write:
\begin{align*}
\int_{\RR^N_+} |\wpsi(x+v,v)|^2\d x 
& = 
\int_{\RR_+}\int_{\RR^{N-1}} |\wpsi((x'+v',x_N+v_N),v)|^2\d x'\d x_N\\
& = 
\int_{-v_N}^\infty \int_{\RR^{N-1}} |\psi(x+v)|^2\d x'\d x_N
+ \int_0^{-v_N} \int_{\RR^{N-1}}   |\wpsi(x+v,v)|^2\d x'\d x_N\\
& = 
\int_{\RR^{N}_+} |\psi(x)|^2\d x
+ \int_0^{-v_N} \int_{\RR^{N-1}}   |\wpsi(x+v,v)|^2\d x'\d x_N.
\end{align*}
Next, we note that if $x\in \RR^{N}_+$ and $x+v\in \RR^{N}_-$, then
$$ \wpsi(x+v,v) = \psi\left(x-\frac{x_N}{v_N}v\right).$$
We deduce
\begin{align*}
\int_0^{-v_N} \int_{\RR^{N-1}}   |\wpsi((x'+v',x_N+v_N),v)|^2\d x'\d x_N  
& = \int_0^{-v_N} \int_{\RR^{N-1}}   |\psi\left(x-\frac{x_N}{v_N}v\right)|^2\d x'\d x_N  \\
& =  -v_N \int_{\RR^{N-1}}   |\psi\left(x',0)\right)|^2\d x'\d x_N  
\end{align*}
and the second inequality in \eqref{eq:L2bd} follows (it is in fact an equality).
\end{proof}

\medskip

\subsection{Convergence of the operators}
%We summarize now the notations introduced in the previous section and which will be used throughout this section and introduce new ones. We will also prove a simple result concerning the extension $\wpsi$ when the domain is a half-space, prove a result on the convergence of $\L^\eps$ to $\L$ and make explicit and rigorous the link between $\L$ and $D^{2s-1}$.  

In this section, we carefully define the operators $D^{2s-1}_\eps$, $\L^\eps$ 
and their limits and prove the main convergence result (Proposition \ref{prop:Lstrong}).

\paragraph{The operators $D^{2s-1}_\eps$ and $D^{2s-1}$.}
For a given function $\psi(x)$, we recall that we defined the operator $D^{2s-1}_\eps$ by \eqref{eq:DDss}.
%(see \eqref{eq:phiepsdef})
%$$
%D^{2s-1}_\eps[\psi](x) = \eps^{1-2s} \int_{\RR^N}   v F(v) [\phi^\eps(x,v) -\psi(x)]\d v
%$$
%where 
%\begin{equation}\label{eq:phi11}
%\phi^\eps (x,v)  =  \int_0^\infty e^{-\nu_0 z} \nu_0 \wpsi (x+\eps v z,v)\d z.
%\end{equation}
%(we note that this operator is not exactly the same as the the operator $\tilde D^{2s-1}_\eps$  defined in \eqref{eq:Deps} since, as we saw above, $\phi^\eps$ is only an approximated solution of \eqref{eq:Phieps}-\eqref{eq:Phibceps}).
After the change of variable $w=vz$, we can also rewrite \eqref{eq:DDss} as
\begin{equation}\label{eq:defDeps0}
D^{2s-1}_\eps[\psi](x) = \eps^{1-2s} \int_{\RR^N}   v F_0(v) [\wpsi (x+\eps v ,v) -\psi(x)]\d v
\end{equation}
where
\begin{align} \label{eq:F0def}
F_0(v)=\int_0^\infty e^{-\nu_0 z} \nu_0   F(v/z) z^{-N-1}\d z.
\end{align} 

We will prove:
\begin{proposition}\label{prop:d2s}
For all functions $\psi(x) \in L^{\infty}(0,\infty;H^1(\Omega))$, we have 
$$
D^{2s-1}_\eps[\psi] \to   D^{2s-1}[\psi] \mbox{ in $L^2(\Omega)$-strong}
$$
as $\eps\to 0$,
where the fractional gradient $D^{2s-1}$  is defined by \eqref{eq:D00} 
or, equivalently, by \eqref{eq:D0}.
%\begin{equation} D^{2s-1} [\psi] (x):= \gamma \nu_0^{1-2s} \Gama(2s-1)
%\int_{\Omega}  (y-x)\cdot \na \psi(y) \frac{y-x}{|y-x|^{N+2s}} \,dy 
%$$
\end{proposition}

%We also define the fractional gradient $D^{2s-1}$  by
%$$
% D^{2s-1} [\psi] (x):= \gamma \nu_0^{1-2s}\Gamma(2s)  \int_{\RR^N}  \big[\wpsi(y,y-x)-\psi(x)\big] \frac{y-x}{|y-x|^{N+2s}} \,dy .
%$$
%We note 

%Using \eqref{eq:psilip}, we see that the right hand side in \eqref{eq:D00} is finite, for instance, if the function $\psi$ is bounded in $\Omega$ and belongs to $C^{\beta}(\Omega)$ for some $\beta>2s-1$ (see also Proposition \ref{prop:DbdP} for conditions to have $ D^{2s-1} [\psi] $ in some $L^p$ space).

\paragraph{The operators $\L^\eps$ and $\L$.}
We recall that $\L^\eps$ is defined by \eqref{eq:defL1'}:
$$
\mathcal L^\eps[\psi](x): = \eps^{-2s}\int_{\RR^N} \nu_0 F(v)   [  \phi^\eps(x,v) -  \psi(x)]  \d v
$$
and using \eqref{eq:phiepsdef} and the change of variable $w=vz$, we find
\begin{align}
\mathcal L^\eps[\psi](x)&  = \eps^{-2s} \nu_0 \int_{\RR^N} \int_0^\infty e^{-\nu_0 z} \nu_0 F(v) [\wpsi (x+\eps v z,v)  -  \psi(x)]  \d v\nonumber\\
& =  \eps^{-2s}\int_{\RR^N} \bF(v)   [ \wpsi (x+\eps v,v) -  \psi(x)]  \d v
\label{eq:defL11} 
\end{align}
where $\bF$ is defined by
\begin{equation}\label{eq:bfdef} 
\bF(v) = \int_0^\infty e^{-\nu_0 z} \nu_0^2 F(v/z) z^{-N}  \d z.
\end{equation}

We also define the corresponding asymptotic operator:
\begin{equation}\label{eq:defL22}
\L[\psi](x)  := \gamma \, \nu_0^{1-2s}\Gamma(2s+1)  \PV \int_{\RR^N}   [ \wpsi (y,y-x) - \psi (x)]\frac{1}{|y-x|^{N+2s}} \d y.
\end{equation}
Since $\wpsi(y,y-x) = \psi(y)$ for $y$ in $\Omega$,  the principal value in the right hand side of \eqref{eq:defL22} is defined for all $x\in\Omega$ if $\psi$ is in $C^{1,\beta}(\Omega)$ for some $\beta>2s-1$.
However, even for such functions,  $\L[\psi](x)$ is typically singular when $x\to\pa \Omega$. 
Indeed, we will prove in Proposition \ref{prop:Linfbd} that if $\psi$ is in $C^{1,\beta}(\Omega)$, then the function $\L[\psi](x)$ remains bounded as $x\to\pa \Omega$ if and only if $\psi$ satisfies the classical Neumann boundary condition
$$ \na_x\psi \cdot n(x)=0.$$

A key result in the proof of Theorem \ref{thm:main} will be the following:
\begin{proposition}\label{prop:Lstrong}
For all function $\psi\in L^\infty(0,\infty;H^2(\Omega))$, such that $\L[\psi]\in L^2(\RR_+\times\Omega)$ we have 
$$  \mathcal{L^\eps} [\psi](t,x)  \longrightarrow  \mathcal{L} [\psi](t,x) \qquad \mbox{ in  $L^2((0,T)\times\Omega) $-strong for all $T>0$.}
$$
\end{proposition}
We refer to Proposition \ref{prop:L2bd} for a characterization of the functions $\psi\in H^2(\Omega)$ such that $\L[\psi](x)\in L^2(\Omega)$).

We recall that we also have (see \eqref{eq:defL2'})
\begin{align*} 
\mathcal L^\eps[\psi](x) %&  = \div_x\left( \eps^{1-2s} \int_{\RR^N}   v F(v) \phi^\eps(x,v) \d v\right)\nonumber\\
&  = \div_x D^{2s-1}_\eps[\psi](x)
%\label{eq:defL23}
\end{align*}
and using Propositions \ref{prop:d2s} and \ref{prop:Lstrong}, we immediately deduce:
\begin{corollary}\label{cor:Ld}
The operator $\L$ defined by \eqref{eq:defL22}  satisfies
$$
\L[\psi](x)  =   \div D^{2s-1} [ \psi] (x)
$$
for all function $\psi\in H^2(\Omega)$ such that $\L[\psi](x)\in L^2(\Omega)$.
\end{corollary}
%Note that this formulation allows us to define $\L[\psi]$   as a distribution in $\mathcal D'(\Omega)$ whenever 
% the function $\psi$ is in $C^{\beta}(\Omega)$ for some $\beta>2s-1$.
We will also prove the following result which justify the formula for $\L$ given in the introduction:
\begin{lemma}\label{lem:opL2}
Let  $\psi\in C^{1,\beta}(\Omega)$ for some $\beta>2s-1$. Then 
\begin{align}
\L[\psi](x) &  = \gamma \nu_0^{1-2s}\Gamma(2s) \PV \int_{\Omega}   \frac{y-x}{|y-x|^{N+2s}}  \cdot \na _y \psi (y)    \d y \qquad \forall x\in\Omega.
\end{align}
\end{lemma}
Finally, integrating by parts the formula \eqref{eq:opL2}, we can also write the following formula for $\L$
\begin{align}
\L[\psi](x) &  = \gamma \nu_0^{1-2s}\Gamma(2s+1) \PV  \int_{\Omega}   \frac{ \psi (y)-\psi(x)  }{|y-x|^{N+2s}} \d y  \nonumber \\
& \qquad  +  \gamma \nu_0^{1-2s}\Gamma(2s)\PV \int_{\pa \Omega}   \frac{y-x}{|y-x|^{N+2s}}\cdot n(y)  [\psi (y)-\psi(x)]    \d y 
\label{eq:opL3} 
\end{align}
which clearly show the relation between the operator $\L$ and the regional fractional laplacian defined in the introduction.

\medskip

We now turn to the proof of these results.

The proof of Propositions \ref{prop:d2s} and  \ref{prop:Lstrong} use very similar computations. 
We will only prove the second one in details, since it is clearly the more complicated of the two.
Before that, we note that the introduction of the functions $F_0$ and $\bF$ above allowed us to eliminate the variable $z$ from the definition of $D^{2s-1}_\eps$ and $\L^\eps$. 
Of course, their behavior for large $v$ is related to that of $F$. More precisely, we have the following Lemma:
\begin{lemma}\label{lem:F0F1}
If $F$ satisfies \eqref{eq:F}, then the functions $F_0$ and $F_1$ defined by \eqref{eq:F0def} and \eqref{eq:bfdef} satisfy:
\begin{equation}\label{eq:F0h}
\left|F_0(v) -  \frac{\gamma_0}{|v|^{N+2s}}\right| \leq \frac{C}{|v|^{N+4s}} \qquad \mbox{ for all } |v|\geq 1, \qquad \gamma_0 = \gamma\, \nu_0^{1-2s} \Gamma (2s)
\end{equation}
% $$\gamma_0 =  \gamma  \int_0^\infty e^{-\nu_0 z}  \nu_0 z^{2s-1} \d z = \gamma\,  \Gamma (2s).$$
%and 
\begin{equation}\label{eq:Fb}
\left|\bF(w) - \frac{\gamma_1 }{|w|^{N+2s}} \right| \leq  \frac{C }{|w|^{N+4s}} \qquad\mbox{ for all } |w|\geq 1, \qquad \gamma_1  = \gamma \nu_0^{1-2s}\Gamma(2s+1)
\end{equation}
\end{lemma}
\begin{proof}[Proof of Lemma \ref{lem:F0F1}]
We only prove \eqref{eq:F0h} since the proof of \eqref{eq:Fb} is almost identical.
We start by noticing that 
$$ 
\int_0^\infty e^{-\nu_0 z} \nu_0   z^{2s-1}\d z=\nu_0^{1-2s} \Gamma (2s)$$
and so
$$
F_0(v) -  \frac{\gamma_0}{|v|^{N+2s}}=\int_0^\infty e^{-\nu_0 z} \nu_0   \left[ z^{-N-2s} F(v/z) - \frac{\gamma} {|v|^{N+2s}} \right] z^{2s-1}\d z.$$
For $|v|\geq 1$, using \eqref{eq:F}, we deduce
\begin{align*}
\left| F_0(v) -  \frac{\gamma_0}{|v|^{N+2s}}\right|
& \leq \int_0^{|v|} e^{-\nu_0 z} \nu_0  \frac{C z^{2s}} {|v|^{N+4s}}  z^{2s-1}\d z+ C \int_{|v|}^\infty e^{-\nu_0 z} \nu_0   z^{2s-1}\d z\\
& \leq  \frac{C  } {|v|^{N+4s}} + e^{-\nu_0|v|/2}
\end{align*}
and the result follows.
\end{proof}

%We now turn to a result of convergence which will be crucial for this paper, as presented in section \ref{sec:formal}, since it is the first step towards passing to the limit in the weak formulation \eqref{eq:weakeps}.

\begin{proof}[Proof of Proposition \ref{prop:Lstrong}]
We denote
\begin{align*}
I_\eps = \underset{\RR_+\times\RR^N_+}{\iint} \big( \mathcal{L}^\eps[\psi](x) - \mathcal{L}[\psi](x) \big)^2 \d t \d x .
\end{align*}
We are going to show that $\lim_{\eps\to 0} I^\eps=0$.
Definition \eqref{eq:defL22}  and the fact that $ \gamma_1  = \gamma \nu_0^{-2s} \Gamma(2s+1)$ (see \eqref{eq:Fb}) imply
\begin{align*}
\L[\psi](x) & =    \PV\int_{\RR^N}   [ \wpsi (y,y-x) - \psi (x)]\frac{\gamma_1}{|y-x|^{N+2s}} \d y \\
& =   \eps^{-2s}    \PV\int_{\RR^N}     [ \wpsi (x+\eps z v,v) -\psi (x)]\frac{\gamma_1}{|v|^{N+2s}} \d v. 
\end{align*}
Using   \eqref{eq:defL11}, we deduce
$$
\L^\eps[ \psi] (x) -  \L[\psi](x) =  \eps^{-2s}\PV\int_{\RR^N}     [\wpsi(x+\eps v ,v)   - \psi(x)]     \left(\bF (v ) -\frac{\gamma_1}{|v|^{N+2s}} \right) \d v  .
$$
%while, recalling that and using the definition \eqref{eq:defL22}:
%\begin{align*}
%\L[\psi](x) & =    \int_{\RR^N}   [ \wpsi (y,y-x) - \psi (x)]\frac{\gamma_1}{|y-x|^{N+2s}} \d y \\
%& =   \eps^{-2s}    \int_{\RR^N}     [ \wpsi (x+\eps z v,v) -\psi (x)]\frac{\gamma_1}{|v|^{N+2s}} \d v. \\
%\end{align*}
Setting $G(v)=\bF(v) -\frac{\gamma_1 }{|v|^{N+2s}} $, we thus write
\begin{align*}
I_\eps& = \underset{\RR_+\times\RR^N_+}{\iint} \left( 
 \eps^{-2s}\PV \int_{\RR^N}  [ \wpsi (x+\eps v,v) -\psi (x)]
G(v)  \d v \right)^2 \d t \d x \\
& \leq  \underset{\RR_+\times\RR^N_+}{\iint} \left( 
 \eps^{-2s}   \PV\int_{|\eps ^\alpha v|<1} [ \wpsi (x+\eps v,v) -\psi(x)]
G(v) \d v \right)^2 \d t \d x \\
& \quad  + \underset{\RR_+\times\RR^N_+}{\iint} \left( 
 \eps^{-2s} \int_{|\eps^\alpha v|>1} [ \wpsi (x+\eps v,v) -\psi(x)]
G(v) \d v \right)^2 \d t \d x \\
& = I_\eps^-+I_\eps^+
\end{align*}
for some $\alpha\in(0,1)$ to be chosen later.

Note that $G(v)$ is singular near $0$, but  $G(v)$ decays faster than $F(v)$ as $|v|\to \infty$.
Indeed, we have (see \eqref{eq:Fb}) $G(v) \leq C |v|^{-(N+4s)}$. 
We  thus write, using \eqref{eq:L2bd}
\begin{align}
I_\eps^+ 
& \leq   \eps^{-4s} \left(\int_{|\eps^\alpha  v|>1} G(v) \d v \right)
\left(\underset{\RR_+\times\RR^N_+}{\iint}  
\int_{|\eps ^\alpha v|>1} [ \wpsi (x+\eps v,v) -\psi(x)]^2
G(v) \, \d v  \d t \d x\right)\nonumber \\
& \leq  
C\eps^{-4s(1-\alpha)} \int_{|\eps^\alpha v|>1} \underset{\RR_+\times\RR^N_+}{\iint}  \big[ | \wpsi (x+\eps v,v)|^2 +|\psi(x)|^2\big]
  \d t \d x\, G(v) \d v\nonumber \\
& \leq  C(\psi) \eps^{-4s(1-\alpha)}
\int_{|\eps^\alpha v|>1} (1+|\eps v|) G(v) \d v\nonumber \\
&\leq C(\psi)\eps^{4s(2\alpha-1)+1-\alpha} \label{eq:Ip}
\end{align}
In order to bound $I_\eps^-$, we  write
\begin{align*} 
\wpsi (x+\eps v,v) -\psi(x) & = \int_0^1 \frac{d}{d\tau } \wpsi(x+\tau \eps v,v)\d \tau \\
& = \int_0^1 \eps v \cdot \na_x \wpsi(x +\tau \eps v,v)\d \tau \\
& = \int_0^{\tau^\eps_0(x,v)}\eps v \cdot \na_x \psi(x+\tau \eps v) \d \tau 
\end{align*}
where we use the exit time $\tau^\eps_0$ defined as
\begin{align}
\tau^\eps_0(x,v) = \sup\{ \tau \in[0,1]\,;\, x+\tau \eps v\in \RR^{N}_+\}. \label{def:tau0eps}
\end{align}
Note that $\tau^\eps_0(x,v)>0$ unless $x\in\pa\RR^N_+$ and $\tau^\eps_0(x,v) =1$ if $x+\eps v\in \RR^{N}_+$, otherwise $\tau_0^\eps(x,v)= -\frac{x_N}{\eps v_N}$.
With one more integration by part, we can also write
\begin{align} 
\wpsi (x+\eps v,v) -\psi(x) 
& = \tau^\eps_0(x,v) \eps v \cdot \na_x \psi(x)  + 
 \int_0^{\tau^\eps_0(x,v)}(\tau^\eps_0(x,v)-\tau) D^2_x \psi(x+\tau \eps v)(\eps v,\eps v) \d \tau \label{eq:hhg} 
\end{align}
We can thus write
\begin{align*}
& \eps^{-2s}\PV \int_{|\eps^\alpha v|<1} [ \wpsi (x+\eps v,v) -\psi(x)]
G(v) \d v  =  \eps^{1-2s}\PV \int_{|\eps^\alpha v|<1} \tau^\eps_0(x,v) v
G(v) \d v  \cdot \na_x \psi(x) \\
&\hspace{3.6cm} +  \eps^{-2s}\int_{|\eps^\alpha v|<1}\int_0^{\tau^\eps_0(x,v)}(\tau^\eps_0(x,v)-\tau) D^2_x \psi(x+\tau \eps v)(\eps v,\eps v) \d \tau
\, G(v) \d v 
\end{align*}
We claim that:
\begin{equation}\label{eq:claim22}
\left| \eps^{1-2s}\PV\int_{|\eps^\alpha v|<1} \tau^\eps_0(x,v) v
G(v) \d v  \cdot \na_x \psi(x)  \right| \leq  \eps^{(\alpha-1)(2s-1)}1_{\{ x_N\leq \eps^{1-\alpha}\}} |\pa_{x_N} \psi(t,x)|.
\end{equation}
Assuming this for now, we deduce (recall that $\tau^\eps_0(x,v)\leq 1$):
\begin{align}
I^-_\eps & \leq 
\underset{\RR_+\times\RR^N_+}{\iint} \left| \eps^{(\alpha-1)(2s-1)} \pa_{x_N} \psi(t,x)\right|^2 1_{\{ x_N\leq \eps^{1-\alpha}\}}  \d t \d x\nonumber \\
& \qquad + 
 \underset{\RR_+\times\RR^N_+}{\iint} \left( 
\eps^{-2s}\int_{|\eps^\alpha v|<1}\int_0^{\tau^\eps_0(x,v)}  |D^2_x \psi(x+\tau \eps v)| |\eps v|^2\d \tau \,
G(v) \d v 
 \right)^2 \d t \d x\nonumber \\
 & \leq 
\underset{\RR_+\times\RR^N_+}{\iint} \left| \eps^{(\alpha-1)(2s-1)} \pa_{x_N} \psi(t,x)\right|^2 1_{\{x_N\leq \eps^{1-\alpha}\}}  \d t \d x\nonumber \\
& \qquad + 
\left( \underset{\RR_+\times\RR^N_+}{\iint}  
\int_{|\eps^\alpha v|<1} \eps^{2-2s} |v|^2 G(v) \d v  \right)
\left( \int_{|\eps^\alpha v|<1} \int_0^{\tau^\eps_0(x,v)}  |D^2_x \psi(x+\tau\eps v)|^2 \d \tau
\eps^{2-2s} |v|^2 G(v) \d v 
  \d t \d x\right) \nonumber \\
   & \leq \eps^{-2(1-\alpha)(2s-1)}
\underset{\RR_+\times\RR^N_+}{\iint} \left|  \pa_{x_N} \psi(t,x)\right|^2 1_{\{x_N\leq \eps^{1-\alpha}\}}  \d t \d x\nonumber\\
& \qquad + C(\eps^{2-2s} + \eps^{(1-\alpha)(2-2s)})^2 
 \underset{\RR_+\times\RR^N_+}{\iint}   |D^2_x \psi(t,x)|^2   \d t \d x \label{eq:hjk}
\end{align}
where we used the fact that $|G(v)|\leq C/|v|^{N+2s}$ and so by definition of $G(v)$
$$\int_{|\eps^\alpha v|<1} \eps^{2-2s} |v|^2 G(v) \d v  \leq C(\eps^{2-2s} + \eps^{(1-\alpha)(2-2s)}).$$

The first term in the right hand side of \eqref{eq:hjk} is not obviously bounded for test functions $\psi$ in $H^2(\Omega)$.
However, we will prove in the next section that this term must go to zero when $\eps\to0$ if $\psi$ is such that $\L[\psi]\in L^2(\RR_+\times\RR^N_+)$ (This is the only place in the proof where we make use of this assumption).
More precisely, 
using Corollary \ref{cor:15} (Equation \eqref{eq:15}), we deduce from \eqref{eq:hjk} that
$$
I^-_\eps\leq C(\psi) \eps^{2(1-\alpha)(2-2s)} + o(1).
$$
Combing this with \eqref{eq:Ip} we get
$$
I^\eps\leq C(\psi) \eps^{4s(2\alpha-1)+ 1-\alpha} + C(\psi) \eps^{2(1-\alpha)(2-2s)}+o(1)
$$
and taking $\alpha\in(1/2,1)$ yields the result.
\medskip

It remains to show \eqref{eq:claim22}.
%To show that the first term is bounded, we must estimate the term $\int_{|\eps v|<1} s^\eps_0(x,v) v
%G(v) \d v $. 
First, we note that $G(v)$ is even and that
 $\tau_0^\eps(x,v',v_N) =\tau_0^\eps(x,-v',v_N) $. We deduce
 $$ \PV\int_{|\eps^\alpha v|<1} \tau^\eps_0(x,v) v_i
G(v) \d v =0\qquad  \mbox{ for all } i=1,\dots N-1.$$
Next, we note that if $|\eps v|\leq x_N$, then $x_n+\eps v_N\geq x_N-|\eps v_N|\geq 0$ and so $\tau_0^\eps(x,v)=1$. We deduce
$$ \PV \int_{|\eps v|<x_N} \tau^\eps_0(x,v) v_N
G(v) \d v =  \PV \int_{|\eps v|<x_N}   v_N
G(v) \d v  =0.$$
Finally, since $|G(v)|\leq C/|v|^{N+2s}$, we have for $x_N < \eps^{1-\alpha}$: 
$$ \left| \int_{\eps^{\alpha-1}x_N<|\eps^\alpha v|<1} \tau^\eps_0(x,v) v_N
G(v) \d v \right| \leq C \eps^{2s-1} |x_N|^{1-2s} + C \eps^{\alpha(2s-1)} \leq C \eps^{\alpha(2s-1)}.$$
The last three equations imply \eqref{eq:claim22}.
\end{proof}

\begin{proof}[Proof of Proposition \ref{prop:d2s}]
Proposition \ref{prop:d2s} is proved in a similar manner. It is simpler of course since it only requires a first order Taylor expansion instead of the second order expansion \eqref{eq:hhg}.
%Note that we do not need \eqref{eq:psineumann}   since the proof of this proposition only requires a first order Taylor expansion instead of the second order expansion \eqref{eq:hhg}. 
We just need to check that \eqref{eq:D00} is equivalent to definition \eqref{eq:D0}. First we note that 
$$
\sum_{i=1}^N \pa_i \left((y_i-x_i) \frac{y_j-x_j}{|y-x|^{N+2s}} \right) = -(2s-1) \frac{y_j-x_j}{|y-x|^{N+2s}}
$$
and so an integration by parts shows that if $\na \psi\in L^\infty \cap L^1(\Omega)$, then
\begin{align*}
\int_{\RR^N}  \big[\wpsi(y,y-x)-\psi(x)\big] \frac{y-x}{|y-x|^{N+2s}} \,dy 
& = (2s-1)
\int_{\RR^N}  (y-x)\cdot \na_y \wpsi(y,y-x) \frac{y-x}{|y-x|^{N+2s}} \,dy \\
& = (2s-1)
\int_{\Omega}  (y-x)\cdot \na \psi(y) \frac{y-x}{|y-x|^{N+2s}} \,dy .
\end{align*}
The result follows by a density argument.
\end{proof}

%We easily deduce a proof of Lemma \ref{lem:Ld}:
%\begin{proof}[Proof of  Lemma \ref{lem:Ld}]
%We recall (see \eqref{eq:defL2'}) that 
%$$\mathcal L^\eps[\psi](x)  = \div_x\left( \eps^{1-2s} \int_{\RR^N}   v F(v) \phi^\eps(x,v) \d v\right).$$
%and so (since $\int vF(v)\d v=0$):
%\begin{align*}
%\mathcal L^\eps[\psi](x) 
%& = \div_x\left( \eps^{1-2s} \int_{\RR^N}   v F(v) [\phi^\eps(x,v)-\psi(x)] \d v\right)\\
%& = \div_x D^{2s-1}_\eps [\psi](x) 
%\end{align*}
%Using Proposition \ref{prop:Lstrong} and Proposition \ref{prop:d2s}, we can pass to the limit in this %equality (in the sense of distribution) and we deduce
%$$ \L [\psi] =    \div_x D^{2s-1}  [\psi](x).$$
%\end{proof}

We end this section with the proof of Lemma \ref{lem:opL2}:
\begin{proof}[Proof of Lemma \ref{lem:opL2}]
The Lemma can be proved directly by computing the divergence in \eqref{eq:D0}. 
Alternatively, 
we can also use the formulation \eqref{eq:defL3} for $\L^\eps$
and \eqref{eq:phiepsdef} to get:
\begin{align*}
\mathcal L^\eps[ \psi] (x) 
&= \eps^{-2s}\int_{\RR^N}  F(v) \eps v \cdot \na_x \phi^\eps(x,v) \d v \\
&= \eps^{1-2s}\int_{\RR^N}  F_0(v)  v \cdot \na_x \wpsi^\eps(x+\eps v,v) \d v \\
&= \int_{\RR^N} \eps^{-N -2s}  F_0\left(\frac{y-x}{\eps}\right) (y-x) \cdot \na_x \wpsi^\eps(y,y-x) dy .
\end{align*}
Since $w\cdot\na_x  \wpsi (y,w) = 0$ for all $w$ whenever $y\notin \Omega$, we can write 
$$
\mathcal L^\eps[ \psi] (x)  =  \int_{\Omega} \eps^{-N -2s}  F_0\left(\frac{y-x}{\eps}\right) (y-x) \cdot \na_x \psi (y) dy .
$$
Proceeding as before, we can pass to the limit in this expression to get \eqref{eq:opL2}.
\end{proof}

\subsection{Control of the boundary correction terms due to definition of the ad hoc test functions} \label{sec:cor}
In this section, we show that the additional terms in \eqref{eq:wftildepsi}-\eqref{eq:wfCorrection}-\eqref{eq:wfinitial} that are due to the corrector $T^\eps$ in the definition of $\phi^\eps_0$
vanish in the limit $\eps\to0$.
We recall that  $T^\eps$ is defined by
\eqref{eq:defT} 
with
\begin{align*}
\bar T^\eps (t,x') & = - \int_{w\cdot n<0}\alpha_0 F(w) (\phi^\eps (t,x,v) -\psi(x))  |w\cdot n | \d w \nonumber \\
& = 
\eps^{2s-1} \alpha_0 D^{2s-1}_\eps [\psi] (t,x) \cdot n 
 \qquad \forall x=(x',0)\in\pa\Omega.
 %\label{eq:defT0}
\end{align*}

%To control the effect of this term in the weak formulation \eqref{eq:wftildepsi}-\eqref{eq:wfCorrection}-\eqref{eq:wfinitial}, we need the following proposition:
The main result is the following proposition:
\begin{proposition}\label{prop:corrector}
There exists a constant $C(\psi)$ (depending on the $L^2$ norms of $\psi$, $\na \psi$, $D^2\psi$ and $\pa_t\psi$)
such that
for any $\eps>0$ and $t\in \ \RR_+$ 
\begin{align}
& \lVert \pa_t T^\eps (t,x) \lVert_{L^2(\Omega)} \leq C(\psi)    \eps ^{\frac{4s-1}{2s+1}}\label{eq:Tnorm2} \\
& \lVert \na _x T^\eps (t,x) \lVert_{L^2(\Omega)} \leq C(\psi)   \eps ^{\frac{4s-1}{2s+1}} .\label{eq:Tnorm3}
\end{align}
Furthermore, all the terms in \eqref{eq:wfCorrection} go to zero when $\eps\to0$.
\end{proposition}
The proof of this proposition relies on the following Lemma, which we prove below:
\begin{lemma} \label{lem:corrector}
There exists a constant $C$ depending on $\|\psi\|_{L^2}$ and $\|\na\psi\|_{L^2}$ such that
for any $\eps>0$ and $t\in \ \RR_+$ we have
\begin{align}\label{eq:TTH}
\left(\int_{\RR^{N}} |T^\eps(t,x)|^2\d x \right)^{1/2} &  \leq C(\psi) \eps ^{\frac{4s-1}{2s+1}}  .
\end{align}
\end{lemma}

\begin{proof}[Proof of Proposition \ref{prop:corrector}]
Since we can differentiate the definition of $ T^\eps$ with respect to $t$ and $x$
to derive bounds on $\pa_t  T^\eps$ and $\na_{x}T^\eps$  similar to  \eqref{eq:TTH}
the estimates \eqref{eq:Tnorm2}, \eqref{eq:Tnorm3} easily follow from \eqref{eq:defT} although the constant in \eqref{eq:Tnorm3} will naturally also depend on the $L^2$ norm of the second derivative of $\psi$.

We now consider the various terms in  \eqref{eq:wfCorrection} one by one. Using the a priori estimate \eqref{eq:energy}, we find:
\begin{align*}
&\underset{\RR_+\times\RR_+^N\times\RR^N}{\iiint} f^{\eps} \pa_t T^\eps(t,x) \chi(v) \d t \d x \d v \\
&\quad \quad \leq \bigg(  \underset{\RR_+\times\RR_+^N\times\RR^N}{\iiint} \frac{|{f^\eps}|^2}{F(v)} \d t \d x \d v \bigg)^{1/2} \bigg(  \underset{\RR_+\times\RR_+^N}{\iint} |\pa_t T^\eps(t,x)|^2 \d t \d x \bigg)^{1/2} \bigg( \underset{\RR^N}{\int} \chi^2(v) F(v) \d v \bigg)^{1/2} \\
&\quad \quad \leq C  \eps ^{\frac{4s-1}{2s+1}} \tto 0
\end{align*}
and  
\begin{align*}
&\underset{\RR_+\times\RR_+^N\times\RR^N}{\iiint} f^{\eps} \eps^{1-2s} v \cdot \na_x T^\eps(t,x) \chi(v) \d t \d x \d v \\
&\quad \quad \leq \eps^{1-2s} \bigg(  \underset{\RR_+\times\RR_+^N\times\RR^N}{\iiint} \frac{|{f^\eps}|^2}{F(v)} \d t \d x \d v \bigg)^{1/2} \bigg(  \underset{\RR_+\times\RR_+^N}{\iint} |\na_x T^\eps(t,x)|^2 \d t \d x \bigg)^{1/2} \bigg( \underset{\RR^N}{\int} |v|^2\chi^2(v) F(v) \d v \bigg)^{1/2} \\
&\quad \quad \leq C \eps^{1-2s}  \eps ^{\frac{4s-1}{2s+1}} = C \eps ^{\frac{4s(1-s)}{2s+1}}\tto 0
\end{align*}
and finally (using the fact that $Q [F] (v) = 0$),
\begin{align*}
&\underset{\RR_+\times\RR_+^N\times\RR^N}{\iiint} f^{\eps}  \eps^{-2s} Q^* \big[ \chi(v) \big] T^\eps(t,x)  \d t \d x \d v \\
&\quad \quad = \underset{\RR_+\times\RR_+^N\times\RR^N}{\iiint} (f^{\eps} -\rho^\eps F) \eps^{-2s} Q^* \big[ \chi(v) \big] T^\eps(t,x)  \d t \d x \d v \\
&\quad \quad \leq \eps^{-2s} \bigg(  \underset{\RR_+\times\RR_+^N\times\RR^N}{\iiint} \frac{(f^\eps -\rho^\eps F)^2}{F(v)} \d t \d x \d v \bigg)^{1/2} \bigg(  \underset{\RR_+\times\RR_+^N}{\iint} |T^\eps(t,x)|^2 \d t \d x \bigg)^{1/2} \bigg( \underset{\RR^N}{\int} Q^* \big[ \chi(v) \big]  F(v) \d v \bigg)^{1/2} \\
&\quad \quad \leq C \eps^{-2s} \eps^s \eps ^{\frac{4s-1}{2s+1}} = C \eps^{\frac{(1-s)(2s-1)}{2s+1}}\tto0  .
\end{align*}

\end{proof}

We complete this section with the proof of Lemma  \ref{lem:corrector}:

\begin{proof}[Proof of Lemma \ref{lem:corrector}]
Using the definition of $ T^\eps$ and \eqref{eq:defDeps0}, we write, for $x=(x',x_N)\in\Omega$:
\begin{align*}%\label{eq:oTe}
 T^\eps(x) & = \alpha_0 \int_{\RR^N_+}   v\cdot n F_0(v) [\psi(  x'+\eps v )-\psi( x')] \d v \, e^{-x_N^2}\nonumber\\ 
&  =  \eps^{2s-1} \alpha_0 \int_{\RR^N_+}   [\psi(x'+y)  -\psi(x')] \frac{1}{\eps ^{N+2s}} F_0 \left(\frac{y}{\eps }\right) y  \cdot n \d y\,  e^{-x_N^2}.
\end{align*}

Furthermore, since $\psi$ satisfies \eqref{eq:limitBCpsi}, we can write:
\begin{align*}
  T^\eps(x) = 
  \eps^{2s-1} \alpha_0   \int_{\RR^N_+}  [\psi(x'+y)  -\psi(x')]  \left[ \frac{1}{\eps ^{N+2s}} F_0 \left(\frac{y}{\eps }\right) -  \frac{ \gamma_0 }{|y|^{N+2s}}\right] y  \cdot n \, e^{-x_N^2} \d y
\end{align*}
and we write $G_0^\eps (y) = \frac{1}{\eps ^{N+2s}} F_0 \left(\frac{y}{\eps }\right) -  \frac{ \gamma_0 }{|y|^{N+2s}}$. \\
To estimate the $L^2$ norm of $T^\eps$, we split the integral with respect to $y$ in two, for $|y|<\eps^{\alpha}$ and $|y|>\eps^{\alpha}$
for some $\alpha\in(0,1)$ to be determined later.
First, we write
\begin{align*}
&\left( \int_{|y|<\eps ^\alpha, y_N>0}   [\psi(x'+y)  -\psi(x')] y  \cdot n  \, G_0^\eps (y) \d y \right)^2\\
&\qquad = \left( \int_{|y|<\eps ^\alpha, y_N>0}   \int_0^1  \na \psi(x'+\tau y) \cdot y \d \tau\, (y  \cdot n)  G_0^\eps (y)\d y\right)^2\\
&\qquad \leq  \int_{|y|<\eps ^\alpha, y_N>0}   \int_0^1  |\na \psi(x'+\tau y)|^2 |y|^2  \left| G_0^\eps (y) \right| \d \tau   \d y  \\
& \qquad \qquad\times  \int_{|y|<\eps ^\alpha}   |y|^2  \left| G_0^\eps (y) \right|   \d y 
\end{align*}
where, since $|v|^{N+2s} F(v)\in L^\infty$ and thus $F_0 (y)\leq  \frac{C}{|y|^{N+2s}}$:
\begin{equation}\label{eq:Fest1} 
 \int_{|y|<\eps ^\alpha}   |y|^2  \left| G_0^\eps (y) \right| \d y\leq   \int_{|y|<\eps ^\alpha}   |y|^2  \frac{C}{|y|^{N+2s}} \d z = C\eps ^{\alpha(2-2s)}.
 \end{equation}
 We thus have:
\begin{align}
&\left( \int_{|y|<\eps ^\alpha, y_N>0}    [\psi(x'+y)  -\psi(x')] y  \cdot n  \,G_0^\eps (y) \d y \right)^2 \nonumber \\
& \qquad\leq C\eps ^{\alpha(2-2s)}  \int_{|y|<\eps ^\alpha, y_N>0}    \int_0^1  |\na \psi(x'+ty)|^2 |y|^2  \left| G_0^\eps (y)\right|  \d t\d y. \label{eq:ikp}
\end{align}
For the integral over $|y|>(\eps z)^\alpha$ we write 
\begin{align*}
& \left( \int_{|y|>\eps ^\alpha, y_N>0}     [\psi(x'+y)  -\psi(x)] y  \cdot n  \,G_0^\eps (y)  \d y\right)^2 \\
&\leq   \left( \int_{|y|>\eps^\alpha} [\psi(x'+y)  -\psi(x)]  |y|  \left| G_0^\eps (y)\right| \d y\right)^2\\
& \leq \left(\int_{|y|>\eps ^\alpha, y_N>0} [\psi(x'+y)  -\psi(x')]^2 |y|  \left| G_0^\eps (y) \right|  \d y  \right) \left(  \int_{|y|>\eps ^\alpha}   |y|  \left| G_0^\eps (y)\right|  \d y\right) .
\end{align*}
Using  \eqref{eq:F0h}, we get
\begin{align*}
\left| G_0^\eps (y) \right| = \bigg| \frac{1}{\eps^{N+2s}} F_0\left(\frac{y}{\eps} \right) -  \frac{ \gamma_0}{|y|^{N+2s}} \bigg| \leq C \frac{\eps ^{2s}}{|y|^{N+4s}}
\end{align*}
for all $|y|\geq \eps ^\alpha$,
we deduce
\begin{align}
\int_{|y|>\eps ^{\alpha}} |y| \left|G_0^\eps (y)\right| \d y &\leq C \eps ^{2s} \int_{|z|>\eps^{\alpha}} |y| \frac{1}{|y|^{N+4s}} \d y \nonumber\\
&\leq C  \eps ^{2s -\alpha (4s-1)} \label{eq:Fest2} 
\end{align}
and so
\begin{align}
& \left( \int_{|y|>\eps ^\alpha, y_N>0}     [\psi(x'+y)  -\psi(x')] y  \cdot n  \,G_0^\eps (y)\d y\right)^2 \nonumber \\
& \qquad\leq  C  \eps ^{2s -\alpha (4s-1)}  \int_{|y|>\eps ^\alpha, y_N>0} [\psi(x'+y)  -\psi(x')]^2 |y|  \left| G_0^\eps (y) \right|  \d y .
\label{eq:hip}
\end{align}

Finally, combing   \eqref{eq:ikp} and \eqref{eq:hip} we get
\begin{align*}
&(\eps^{2s-1} \alpha_0)^{-2} \int_{\RR_+^{N}} |T^\eps(x)|^2\d x \\
&\qquad \leq  C  \eps^{\alpha(2-2s)} \int_{x_N>0} e^{-2x_N^2} \int_{\RR^{N-1}}  \int_{|y|<\eps ^\alpha, y_N>0}  \int_0^1  |\na \psi(x'+\tau y)|^2 |y|^2  \left| G_0^\eps (y)\right|\d \tau \d y\d x' \d x_N\\
&\qquad\quad  + C \eps ^{2s -\alpha (4s-1)} \int_{x_N>0} e^{-2x_N^2} \int_{\RR^{N-1}}  \int_{|y|> \eps^\alpha, y_N>0} [\psi(x'+y)  -\psi(x')] ^2 |y|  \left| G_0^\eps (y)\right|\d y\d x' \d x_N \\
& \qquad\leq  C\eps^{\alpha(2-2s)} \frac{\sqrt{\pi}}{2\sqrt{2}}   \int_{|y|<\eps ^\alpha, y_N>0}  \int_0^1   \int_{\RR^{N-1}} |\na \psi(x'+\tau y)|^2\d x' |y|^2  \left| G_0^\eps (y)\right|\d \tau \d y \\
&\qquad\quad  + C  \eps^{2s -\alpha (4s-1)} \frac{\sqrt{\pi}}{2\sqrt{2}}  \int_{|y|>\eps  ^\alpha, y_N>0} \int_{\RR^{N-1}}[\psi(x'+y)  -\psi(x')] ^2 \d x' |y|  \left| G_0^\eps (y)\right|\d y \\
&\qquad \leq  C(\na \psi) \eps^{\alpha(2-2s)}    \int_{|y|<(\eps z)^\alpha, y_N>0}      |y|^2  \left| G_0^\eps (y)\right|\d y  \\
&\qquad\quad  + C (\psi)  \eps^{2s -\alpha (4s-1)}  \int_{|y|>\eps ^\alpha, y_N>0}   |y|  \left| G_0^\eps (y)\right|\d y 
\end{align*}
where, up to constants,
$$C(\psi) = \int_{\RR^{N-1}} |\psi(x')|^2\d x' \text{ and} \quad C(\na \psi) = \int_{\RR^{N-1}} |\na \psi(x')|^2\d x' .$$
Using \eqref{eq:Fest1} and \eqref{eq:Fest2}, we deduce
\begin{align*}
\left(\int_{\RR^{N-1}} |T^\eps(x)|^2\d x\right)^{1/2} & \leq C(\psi)(\eps^{2s-1} \alpha_0) \left[\eps^{\alpha(2-2s)}  +    \eps^{2s -\alpha (4s-1)}  \right]
\end{align*}
and we see that we need to take $\alpha=\frac{2s}{1+2s}$ to get \eqref{eq:TTH}.
%\begin{align*}
%\left(\int_{\RR^{N-1}} |\overline T^\eps(x')|^2\d x'\right)^{1/2} &  \leq C(\psi) \eps^{\frac{4s(1-s)}{1+2s}}  .
%\end{align*}
%We deduce (see \eqref{eq:defT}):
%$$ 
%\left(\int_{\RR^{N}}(1+x_N)^2 | T^\eps(x)|^2\d x\right)^{1/2}  \leq C(\psi) \eps^{2s-1} \eps^{\frac{4s(1-s)}{1+2s}}
%= C(\psi) \eps ^{\frac{4s-1}{2s+1}}
%$$
%and the proof is complete.
\end{proof}

\subsection{Derivation of the asymptotic equation} \label{sec:derivation}
We can now complete the proof of Theorem \ref{thm:main}.
Recall that we only need to pass to the limit in \eqref{eq:wftildepsi}-\eqref{eq:wfCorrection}-\eqref{eq:wfinitial}.
We proved in Section
\ref{sec:cor} above that \eqref{eq:wfCorrection} vanish in the limit.

Furthermore, the weak convergence of $\rho^\eps$ (Lemma \ref{lem:apriori}) and the strong convergence of $\L^\eps[\psi]$ (Proposition \ref{prop:Lstrong}) immediately implies
\begin{proposition}\label{prop:lim2}
For all $\psi\in L^\infty(0,\infty; H^2(\Omega))$, satisfying \eqref{eq:psineumann} the following limit hold:
$$
\lim_{\eps\to0 }\underset{\RR_+\times\RR_+^N }{\iint} \rho^\eps  \mathcal L^\eps[\psi] \d t \d x  = \underset{\RR_+\times\RR_+^N}{\iint} \rho \,\mathcal L[ \psi] \d t \d x
$$
where the operator $\L$ is defined by \eqref{eq:defL22}.
\end{proposition}

The convergence of the last two terms (the time derivative and the initial condition term) will follow from the weak convergence of $f^\eps$ to $\rho F$ and the
following Lemma
applied to $\phi^\eps(0,x,v)$ (for the initial condition \eqref{eq:wfinitial}) and to $\pa_t\phi^\eps$ (for the  first term in \eqref{eq:wftildepsi}):
\begin{lemma}\label{lem:L20}
For all test function $\psi \in L^\infty(0,\infty; H^1(\Omega))$, the following holds:
$$
\lim_{\eps\to 0}\underset{\RR_+\times\RR_+^N\times\RR^N}{\iiint}     \big|\phi^\eps(t,x,v) -  \psi(t,x) \big|^2 F(v)   \d t \d x \d v =0
$$
\end{lemma}

%We now turn our attention to the main terms in the weak formulation, i.e. \eqref{eq:wftildepsi}.
%We begin with the first term i.e. the following proposition:
%\begin{proposition}\label{prop:lim}
%For all $\psi \in W^{1,\infty}([0,\infty); H^1(\Omega))$, the following limit holds:
%\begin{equation}\label{eq:lim1}
%\lim_{\eps\to0 }\underset{\RR_+\times\RR_+^N\times\RR^N}{\iiint} f^\eps \pa_t \phi^\eps \d t \d x \d v  = 
%\underset{\RR_+\times\RR_+^N}{\iint} \rho \pa_t \psi  \d t \d x
%\end{equation}
%\end{proposition}

%\begin{proof}[Proof of Proposition \ref{prop:lim}]
%To prove \eqref{eq:lim1}, we  write
%\begin{align*}
%\underset{\RR_+\times\RR_+^N\times\RR^N}{\iiint} f^\eps (t,x,v) \pa_t\phi^\eps(t,x,v) \d t \d x \d v &=  \underset{\RR_+\times\RR_+^N}{\iint} f^\eps(t,x,v) \pa_t \psi(t,x) \d t \d x \\
%&+ \underset{\RR_+\times\RR_+^N\times\RR^N}{\iiint} f^\eps(t,x,v) \big( \pa_t \phi^\eps(t,x,v) - \pa_t \psi(t,x) \big) \d t \d x \d v .
%\end{align*}
%The first term converges as need, since $f^\eps$ converges $L^2$-weak to $\rho(t,x) F(v)$.
%The second term is bounded by
%$$ 
%\left(\underset{\RR_+\times\RR_+^N\times\RR^N}{\iiint} \frac{|f^\eps|^2 }{F(v)} \d t \d x \d v\right)^{1/2} \times 
%\left(\underset{\RR_+\times\RR_+^N\times\RR^N}{\iiint}     \big| \pa_t \phi^\eps(t,x,v) - \pa_t \psi(t,x) \big|^2 F(v)   \d t \d x \d v\right)^{1/2} ,
%$$
%and the following lemma applied to $\pa_t \psi$ yields \eqref{eq:lim1}. 
%\end{proof}

\begin{proof}[Proof of Lemma \ref{lem:L20}]
To prove the lemma,  we first write
\begin{align*}
|\phi^\eps(t,x,v) -  \psi(t,x) |^2 & = \left| \int_0^\infty e^{-\nu_0 z} \nu_0 \big(\wpsi(t,x+\eps v z,z) -\psi(t,x)\big) \d z\right|^2 \\
& \leq  \int_0^\infty e^{-\nu_0 z}\nu_0 \big|\wpsi(t,x+\eps v z,v) -\psi(t,x)\big|^2\d z
\end{align*}
and so
\begin{align*}
& \underset{\RR_+\times\RR_+^N\times\RR^N}{\iiint}     \big| \phi^\eps(t,x,v) -  \psi(t,x) \big|^2 F(v)   \d t \d x \d v \\
& \qquad \leq  
\underset{\RR_+\times\RR_+^N\times\RR^N}{\iiint} 
\int_0^\infty e^{-\nu_0 z}\nu_0 |\wpsi(t,x+\eps v z,v) -\psi(t,x)|^2\,
F(v)   \d z \d w   \d x \d t \\
& \qquad \leq  
\underset{\RR_+\times\RR_+^N\times\RR^N}{\iiint} 
\int_0^\infty e^{-\nu_0 z}\nu_0 |\wpsi(t,x+\eps w,w) -\psi(t,x)|^2\, 
F(w/z) z^{-N}  \d z \d w   \d x \d t \\
& \qquad \leq  
\underset{\RR_+\times\RR_+^N\times\RR^N}{\iiint} 
 |\wpsi(t,x+\eps w,w) -\psi(t,x)|^2\, \bF (w)
 \d w   \d x \d t 
\end{align*}
where $\bF$ is given by \eqref{eq:bfdef}. 

We now write
\begin{align*}
& \underset{\RR_+\times\RR_+^N\times\RR^N}{\iiint}     \big| \wpsi(t,x+\eps v, \eps v) -  \psi(t,x) \big|^2 \bF(v)   \d t \d x \d v \\
& \qquad = 
 \int_{|\eps v|< 1} \underset{\RR_+\times\RR_+^N }{\iint}     \big| \wpsi(t,x+\eps v, \eps v) -  \psi(t,x) \big|^2 \bF(v)   \d t \d x \d v  \\
& \qquad\quad + 
\int_{|\eps v |>1}  \underset{\RR_+\times\RR_+^N }{\iint}     \big| \wpsi(t,x+\eps v, \eps v) -  \psi(t,x) \big|^2 \bF(v)   \d t \d x \d v 
\end{align*}
To  bound the integral over $|\eps v |<1$ we take advantage of the regularity of $ \psi$ to write, using Taylor and $\tau_0^\eps$ defined in \eqref{def:tau0eps}:
\begin{align*}
\int_{|\eps v| < 1} &\underset{\RR_+\times\RR_+^N}{\iint} \big( \wpsi(t,x+\eps v,\eps v) - \psi(t,x) \big)^2 \bF(v) \d t \d x \d v \\
&= \int_{|\eps v|< 1} \underset{\RR_+\times\RR_+^N}{\iint} \bigg( \int_0^{\tau_0^\eps(x,v)} \eps v \cdot \na_x  \psi(t,x+\tau \eps v)\d\tau \bigg)^2 \bF(v) \d t \d x \d v \\
& \leq   \int_0^1 \int_{|\eps v|< 1} |\eps v|^2  \bF(v)  \underset{\RR_+\times\RR_+^N}{\iint} |\na_x  \psi(t,x+\tau \eps v)|^2 1_{\{\tau\leq \tau_0^\eps(x,v)\}} \d t\d x \d v\d \tau\\
& \leq   \int_0^1 \int_{|\eps v|< 1} |\eps v|^2  \bF(v) C(\psi)  \d v\d \tau\\
& \leq   C(\psi) \eps^{2s}
\end{align*}
And for the integral over $|\eps v |>1$ we use the decay of $F$ to write (using \eqref{eq:L2bdt}):
\begin{align*}
\int_{|\eps v|> 1} &\underset{\RR_+\times\RR_+^N}{\iint} \big( \wpsi(t,x+\eps v,\eps v) - \psi(t,x) \big)^2 \bF(v) \d t \d x \d v \\
&\leq \int_{|\eps v|>1}C(\psi) (1+\eps |v|) \bF(v) \d v \\
&\leq C(\psi)\eps^{2s} \int_{|w|>1}(1+|w|)\frac{C}{|w|^{N+2s}}\d w  .
\end{align*}
This completes the proof of Lemma \ref{lem:L20}.
\end{proof}
\medskip

\section{Well posedness of the asymptotic equation} \label{sec:Lwellposed}

\subsection{A functional framework for $D^{2s-1}$}\label{app:D}

Using the expression of $D^{2s-1}$ in terms of the extension $\wpsi$, we can improve on Proposition \ref{prop:DbdP} and prove the following result, which will be useful in the Proof of Theorem \ref{thm:evolution}:
\begin{proposition}\label{prop:D}
If $\psi \in H^{2s-1+\beta}(\Omega)$ for some $\beta>0$, then $D^{2s-1}[\psi] \in (L^2(\Omega))^N$.
\end{proposition}
In particular, since $s>2s-1$ when $s<1$, we deduce that if $\psi\in H^s(\Omega)$, then $D^{2s-1}[\psi] \in L^2(\Omega)$.

\begin{proof}
We recall the definition \eqref{eq:D00} of $D^{2s-1}$:
\begin{align*}
 D^{2s-1} [\psi] (x) & = \gamma \nu_0^{1-2s}\Gamma(2s)  \int_{\RR^N}  \big[\wpsi(y,y-x)-\psi(x)\big] \frac{y-x}{|y-x|^{N+2s}} \,dy \\
 & =   \gamma \nu_0^{1-2s}\Gamma(2s)  \int_{\Omega}  \big[\psi(y)-\psi(x)\big] \frac{y-x}{|y-x|^{N+2s}} \,dy\\
 & \qquad + \gamma \nu_0^{1-2s}\Gamma(2s)  \int_{\RR^N\setminus \Omega}  \big[\wpsi(y,y-x)-\psi(x)\big] \frac{y-x}{|y-x|^{N+2s}} \,dy
 \end{align*}
Recalling that $\Omega$ is the upper-half space $y_N>0$, we can use \eqref{eq:fhhr} to write 
$$\wpsi(y,y-x)=\psi(y'-\frac{y_N}{y_N-x_N}(y'-x'),0) \qquad \forall x\in\Omega, \; y\in\RR^N\setminus\Omega.$$
%\begin{align*}
%\int_{\RR^N\setminus \Omega}  \big[\wpsi(y,y-x)-\psi(x)\big] \frac{y-x}{|y-x|^{N+2s}} \,dy
%& = \int_{\RR^N\setminus \Omega}  \big[\psi(y'-\frac{y_N}{y_N-x_N}(y'-x'),0)-\psi(x)\big] \frac{y-x}{|y-x|^{N+2s}} \,dy
%\end{align*}
We now do a change of variable $z'=y'-\frac{y_N}{y_N-x_N}(y'-x')$. 
Denoting $z=(z',0)$, we check that 
$$ y-x = \frac{x_N-y_N}{x_N} (z-x)$$ and so
\begin{align*}
\int_{\RR^N\setminus \Omega}  \big[\wpsi(y,y-x)-\psi(x)\big] \frac{y-x}{|y-x|^{N+2s}} \,dy
& =\int_{\RR^{N-1}} \int_{-\infty}^0 
 \big[\psi(z',0)-\psi(x)\big]  \frac{z-x}{|z-x|^{N+2s}}  \left( \frac{x_N-y_N}{x_N}\right)^{-2s}\,dy_N \,dz' \\
 & = \frac{x_N}{2s-1}\int_{\pa\Omega} \big[\psi(z)-\psi(x)\big]  \frac{z-x}{|z-x|^{N+2s}} dS(z)
 \end{align*}
We deduce
\begin{align} 
 D^{2s-1} [\psi] (x)  
 & =   \gamma \nu_0^{1-2s}\Gamma(2s)  \int_{\Omega}  \big[\psi(y)-\psi(x)\big] \frac{y-x}{|y-x|^{N+2s}} \,dy\nonumber \\
 & \qquad + \gamma \nu_0^{1-2s}\Gamma(2s-1) x_N  \int_{\pa\Omega} \big[\psi(z)-\psi(x)\big]  \frac{z-x}{|z-x|^{N+2s}} dS(z) .\label{eq:Dlast}
 \end{align}

Note that the last term can also be split as
\begin{align*}
x_N  \int_{\pa\Omega} \big[\psi(z)-\psi(x)\big]  \frac{z-x}{|z-x|^{N+2s}} dS(z) 
 & =x_N\int_{\pa\Omega} \big[\psi(z',0)-\psi(x',0)\big]  \frac{z-x}{|z-x|^{N+2s}} dS(z) \\
 & \qquad +x_N\int_{\pa\Omega} \big[\psi(x',0)-\psi(x',x_N)\big]  \frac{z-x}{|z-x|^{N+2s}} dS(z) .
\end{align*}
where
$$x_N\int_{\RR^{N-1}} \frac{z_i-x_i}{|z-x|^{N+2s}} \, dz'= 
\begin{cases}
0 & \mbox{ if } i=1,\dots,N-1 \\
-x_N^{1-2s} \int_{\RR^{N-1}}\frac{1}{(z'^2+1)^{\frac{N+2s}{2}}}\, dz' & \mbox{ if } i=N
\end{cases}.
$$
We thus need to show that the three terms
\begin{align*}
I_1 & = \int_{\Omega}  \frac{|\psi(y)-\psi(x)|}{|y-x|^{N+2s-1}} \,dy\\
I_2& = x_N\int_{\pa\Omega} \frac{|\psi(z',0)-\psi(x',0)|} {|z-x|^{N+2s-1}} dS(z) \\
I_3 & = x_N^{1-2s} |\psi(x',0)-\psi(x',x_N)|
\end{align*}
are bounded in $L^2(\Omega)$ by $\|\psi\|_{H^{2s-1+\beta}(\Omega)}$.
The first term is obvious and the last follows from Hardy's inequality.
For the second term, we write
\begin{align*}
& \int_\Omega |I_2(x)|^2\, dx \\
 & \leq 
\int_\Omega x_N^2\left(\int_{\RR^{N-1}} \frac{|\psi(z',0)-\psi(x',0)|} {|(z'-x')^2+x_N^2|^{\frac{N+2s-1}{2}}} dz' \right)^2\, dx\\
& \leq 
\int_{\Omega\cap\{x_N<1\}} x_N^2 \int_{\RR^{N-1}} \frac{|\psi(z',0)-\psi(x',0)|^2} {|z'-x'|^{N+4s-4+\beta'}} dz'  \, \int_{\RR^{N-1}} \frac{|z'-x'|^{N+4s-4+\beta'}} {|(z'-x')^2+x_N^2|^{N+2s-1}} dz' dx\\
& \quad + \int_{\Omega\cap\{x_N>1\}} x_N^2 \int_{\RR^{N-1}} \frac{|\psi(z',0)-\psi(x',0)|^2} {|(z'-x')^2+x_N^2|^{\frac{N-1+2\beta}{2}}} dz'\int_{\RR^{N-1}} \frac{1} {|(z'-x')^2+x_N^2|^{\frac{N+4s-1-2\beta}{2}}} dz' \, dx\\
& \leq 
\int_{\Omega\cap\{x_N<1\}} x_N^{-1+\beta'} \int_{\RR^{N-1}} \frac{|\psi(z',0)-\psi(x',0)|^2} {|z'-x'|^{N+4s-4+\beta'}} dz'  \, dx\int_{\RR^{N-1}} \frac{|z'|^{N+4s-4}} {|z'^2+1|^{N+2s-1}} dz'  \\
& \quad + C\int_1^\infty \int_{\RR^{N-1}} x_N^{2-4s+2\beta} \int_{\RR^{N-1}} \frac{|\psi(z',0)-\psi(x',0)|} {|(z'-x')^2+x_N^2|^{\frac{N-1+2\beta}{2}}} dz' \, dx'\, dx_N \int_{\RR^{N-1}} \frac{1} {|z'^2+1^2|^{\frac{N+4s-1-2\beta}{2}}} dz'  \\
& \leq C \int_{\RR^{N-1}}\int_{\RR^{N-1}} \frac{|\psi(z',0)-\psi(x',0)|^2} {|z'-x'|^{N+4s-4+\beta'}} dz'\, dx'\\ 
& \quad + C \int_1^\infty \int_{\RR^{N-1}} x_N^{2-4s+2\beta} \int_{\RR^{N-1}} \frac{|\psi(z',0)-\psi(x',0)|^2} {|(z'-x')^2+x_N^2|^{\frac{N-1+2\beta}{2}}} dz' \, dx'\, dx_N 
%& \leq 
%C \int_\Omega x_N^{2-2s} \int_{\RR^{N-1}} \frac{|\psi(z',0)-\psi(x',0)|^2} {|(z'-x')^2+x_N^2|^{\frac{N+2s-1}{2}}} dz'  \, dx
\end{align*}
where
\begin{align*}
& \int_1^\infty \int_{\RR^{N-1}} x_N^{2-4s+2\beta} \int_{\RR^{N-1}} \frac{|\psi(z',0)-\psi(x',0)|^2} {|(z'-x')^2+x_N^2|^{\frac{N-1+2\beta}{2}}} dz' \, dx'\, dx_N \\
& \qquad  \leq 2 \int_1^\infty  x_N^{2-4s+2\beta}\int_{\RR^{N-1}} \int_{\RR^{N-1}} \frac{|\psi(z',0)|^2} {|(z'-x')^2+x_N^2|^{\frac{N-1+2\beta}{2}}} dz' \, dx'\, dx_N.
\end{align*}
The result follows.

\end{proof}

\subsection{Proof of Theorem \ref{thm:evolution}}

The proof of Theorem \ref{thm:evolution} relies on Hille-Yoshida theorem. The first step, which will occupy most of this section is thus devoted to the proof of  the  well-posedness of the stationary problem:
\begin{equation}\label{eq:Neumanns}
\left\{
\begin{array}{ll}
\vphi (x)- \mathcal L [\vphi](x)= g(x) & \mbox{ for all  } x\in \Omega, \\
D^{2s-1} [\vphi](x) \cdot n(x) =0 & \mbox{ for all } x\in \pa\Omega.
\end{array}
\right.
\end{equation}
Using \eqref{eq:ipp1}, we see that classical solutions of \eqref{eq:Neumanns} satisfy
\begin{equation}\label{eq:weak00}
\int_{\Omega} \vphi(x) \psi(x)\d x +    \int_{\Omega} D^{2s-1}[\vphi](x)\cdot \na \psi(x) \d x = \int_\Omega g(x)\psi(x)\d x 
\end{equation}
for all $\psi\in\mathcal D(\overline \Omega)$, and using \eqref{eq:ipp3}, we can write this as
\begin{equation}\label{eq:weak00'}
\int_{\Omega} \vphi(x) \psi(x)\d x +   \gamma \nu_0^{1-2s} \Gamma(2s-1)
 \int_{\Omega}   \int_\Omega (y-x)\cdot \na \vphi(y) (y-x) \cdot \na \psi(x)\frac{dy\, dx}{|y-x|^{N+2s}}   = \int_\Omega g(x)\psi(x)\d x 
\end{equation}

We thus introduce the following bilinear symmetric form:
\begin{equation}\label{eq:a} 
a(\vphi,\psi) = \int_{\Omega} \vphi(x) \psi(x)\d x +  \gamma \nu_0^{1-2s} \Gamma(2s-1)
 \int_{\Omega}   \int_\Omega (y-x)\cdot \na \vphi(y) (y-x) \cdot \na \psi(x)\frac{dy\, dx}{|y-x|^{N+2s}} .
\end{equation}
This form is bilinear and symmetric. It is clearly well defined for instance if $\vphi$ and $\psi$ are in $H^1(\Omega)$, but we are going to show that is can be extended to the space $ H^s(\Omega)$.

Indeed, we will show  the following proposition:
\begin{proposition}\label{prop:norm}
The bilinear form $a(\vphi,\psi)$ defined by \eqref{eq:a} satisfies
\begin{equation}\label{eq:aa}
\begin{cases}
a(\vphi,\psi)\leq C \| \vphi\|_{H^s}  \| \psi\|_{H^s} & \mbox{ for all } \vphi,\, \psi\,  \in H^1(\Omega) \\
a(\vphi,\vphi)\geq c \| \vphi\|^2_{H^s} &  \mbox{ for all } \vphi \in H^1(\Omega) 
\end{cases}
 \end{equation}
for some constants $c$ and $C$ depending only on $\Omega$ and $s$
and can thus be extended into a bilinear continuous form on $H^s(\Omega) \times H^s(\Omega)$.
\end{proposition}

Lax-Milgram's Theorem then implies the existence of a weak solution to \eqref{eq:Neumanns}. More precisely, we have:
\begin{theorem} \label{thm:weak}
For all $g$ in $L^2(\Omega)$, there exists a unique $\vphi\in H^s(\Omega)$ such that
$$ a(\vphi,\psi) = \int_\Omega  g(x)\psi(x)\d x \qquad \forall \psi\in H^s(\Omega).$$
\end{theorem}

We now turn to the proof of Proposition \ref{prop:norm}, which relies on the following lemma, the proof of which is postponed until after the proof of Proposition \ref{prop:norm}:
\begin{lemma}\label{lem:a}
For all $\vphi\in H^1(\Omega)$, there holds: 
\begin{align}
 \int_{\Omega} D^{2s-1}[\vphi]\cdot \na \vphi \d x 
 %& =  2s 
%\int_{\RR^N} \int_{\RR^N} \frac{[\wvphi(x,x-z) -\wvphi(z,z-x)] ^2}{|z-x|^{N+2s}} \d z\d x \label{eq:a1} \\
& =s \gamma \Gamma(2s) \int_\Omega \int_\Omega \frac{[\vphi(x)-\vphi(y)]^2}{|x-y|^{N+2s}}\d x\d y \nonumber \\
& \quad +\gamma \Gamma(2s)   \int_\Omega \int_{\pa\Omega}
[\vphi(x)-\vphi(y)]^2 \frac{y-x}{|y-x|^{N+2s}}\cdot n(y)\d S(y)\d x.
\label{eq:a2}
\end{align}
\end{lemma}

\begin{proof}[Proof of Proposition \ref{prop:norm}]
We just need to prove that 
\begin{equation}\label{eq:aaa}
 c  \| \vphi\|^2_{H^s} \leq a(\vphi,\vphi)\leq C \| \vphi\|^2_{H^s} \qquad  \mbox{ for all } \vphi \in H^1(\Omega) \end{equation}
since Cauchy-Schwarz inequality then gives
$$ a(\vphi,\psi)^2 \leq a(\vphi,\vphi) a(\psi,\psi) \leq C \| \vphi\|^2_{H^s}  \| \psi\|^2_{H^s} .$$
 
In order to prove \eqref{eq:aaa}, we first note that
both terms in   \eqref{eq:a2} are non-negative and so
we immediately get
\begin{align*}
 a(\vphi,\vphi)& \geq \int_{\Omega} |\vphi(x) |^2\d x + s \gamma \Gamma(2s) \int_\Omega \int_\Omega \frac{[\vphi(x)-\vphi(y)]^2}{|x-y|^{N+2s}}\d x\d y \\
 & \geq C \| \vphi\|^2_{H^s}
 \end{align*}

To prove the other inequality, we need to show that the last term in \eqref{eq:a2} can be bounded by 
$ \| \vphi\|^2_{H^s}$.
First, we write
\begin{align*}
  \int_\Omega \int_{\pa\Omega}
[\vphi(x)-\vphi(y)]^2 \frac{y-x}{|y-x|^{N+2s}}\cdot n(y)\d S(y)\d x
 &
= \int_\Omega \int_{\RR^{N-1}}
[\vphi(x',x_N)-\vphi(y',0)]^2 \frac{x_N}{|y-x|^{N+2s}}\d y'\d x   \\
&\leq  \int_\Omega \int_{\RR^{N-1}}
[\vphi(x',x_N)-\vphi(x',0)]^2 \frac{x_N }{|y-x|^{N+2s}}\d y'\d x \\
& \quad + \int_\Omega \int_{\RR^{N-1}}
[\vphi(x',0)-\vphi(y',0)]^2 \frac{x_N }{|y-x|^{N+2s}}\d y'\d x 
\end{align*}
Using the fact that
$$
\int_{\RR^{N-1}} \frac{x_N }{|y-x|^{N+2s}}\d y' = \frac{1}{x_N^{2s}} \int_{\RR^{N-1}} \frac1{|1+|z|^2|^{\frac{N+2s}{2}}}\d z
$$
and 
$$
\int_0^\infty \frac{x_N }{|y-x|^{N+2s}}\d x_N = \frac{1}{|y'-x'|^{N+2s-2}} \int_0^\infty \frac{t}{(1+t^2)^{\frac{N+2s}{2}}} \d t
$$
we deduce
\begin{align}
  \int_\Omega \int_{\pa\Omega}
[\vphi(x)-\vphi(y)]^2 \frac{y-x}{|y-x|^{N+2s}}\cdot n(y)\d S(y)\d x
&\leq C  \int_\Omega 
 \frac{[\vphi(x',x_N)-\vphi(x',0)]^2 }{x_N^{2s}}\d x\nonumber  \\
& \quad +  C \int_{\RR^{N-1}}  \int_{\RR^{N-1}}
 \frac{[\vphi(x',0)-\vphi(y',0)]^2 }{|y-x|^{N+2s-2}}\d y'\d x'. \label{eq:bd2}
\end{align}
The second term in \eqref{eq:bd2} is bounded by a Sobolev trace theorem (note that $N+2s-2=(N-1)+2(s-1/2)$). Indeed, we recall the following theorem:
\begin{theorem}[\cite{DiNezzaPalatucciValdinoci12}]
For all $\vphi\in H^s(\Omega)$, we have
$$ 
\| \vphi \| _{H^{s-1/2}(\pa\Omega)} \leq C\| \vphi\|_{H^s(\Omega)}.
$$
\end{theorem}
We deduce
\begin{equation}\label{eq:bd3}
 \int_{\RR^{N-1}}  \int_{\RR^{N-1}}
 \frac{[\vphi(x',0)-\vphi(y',0)]^2 }{|y-x|^{N+2s-2}}\d y'\d x'\leq C\| \vphi\|^2_{H^s(\Omega)} .
\end{equation}
In order to bound the first term in the right hand side of \eqref{eq:bd2}, we use the fractional Hardy inequality (Theorem \ref{thm:hardy}). 
Since the function $x\mapsto \vphi(x',x_N)-\vphi(x',0)$ is in $H^s_0(\Omega)$, 
and $s\in(1/2,1)$, we have:
\begin{equation}\label{eq:bd4} 
\int_\Omega 
 \frac{[\vphi(x',x_N)-\vphi(x',0)]^2 }{x_N^{2s}}\d x\leq C \| \vphi\|^2_{H^s(\Omega)}.
\end{equation}

Combining \eqref{eq:bd2}, \eqref{eq:bd3} and \eqref{eq:bd4}, we deduce
$$ 
 \int_\Omega \int_{\pa\Omega}
[\vphi(x)-\vphi(y)]^2 \frac{y-x}{|y-x|^{N+2s}}\cdot n(y)\d S(y)\d x\leq C \| \vphi\|^2_{H^s(\Omega)}
$$
and \eqref{eq:a2} yields
$$ a(\vphi,\vphi)\leq  C \| \vphi\|^2_{H^s(\Omega)}$$
which gives \eqref{eq:aaa} and concludes the proof.
\end{proof}

\begin{proof}[Proof of Lemma \ref{lem:a}]
We use the approximated operator $D_\eps^{2s-1}$ to prove this equality. 
First, we write, for $\vphi\in D(\overline \Omega)$ (using the definition \eqref{eq:defDeps0} for $D_\eps^{2s-1}$):
\begin{align*}
 \int_{\Omega} D_\eps^{2s-1}[\vphi]\cdot \na \vphi \d x & = 
\eps^{1-2s} \int_{\Omega} \int_{\RR^N}F_0(v) [\wvphi(x+\eps v,v) -\vphi(x)] v \cdot \na_x \vphi(x) \d v\d x \\
& = \eps^{-N-2s}  \int_{\Omega} \int_{\RR^N}F_0\left(\frac{y-x}{\eps}\right) [\wvphi(y,y-x) -\vphi(x)] (y-x) \cdot \na_x \vphi(x) \d y\d x 
\end{align*}
Since $v\cdot \na_v \wvphi(x,v) = 0$ for all $x$ and $v$, we can write
\begin{align*}
& \int_{\Omega} D_\eps^{2s-1}[\vphi]\cdot \na \vphi \d x  \\
&\qquad  = \eps^{-N-2s}  \int_{\Omega} \int_{\RR^N}F_0\left(\frac{y-x}{\eps}\right) [\wvphi(y,y-x) -\vphi(x)] (y-x) \cdot \na_x [ \vphi(x)-\wvphi(y,y-x)] \d y\d x \\
&\qquad  = -\frac 1 2 \eps^{-N-2s}  \int_{\Omega} \int_{\RR^N}F_0\left(\frac{y-x}{\eps}\right)(y-x) \cdot \na_x [ \vphi(x)-\wvphi(y,y-x)]^2 \d y\d x \\
&\qquad  = \frac 1 2 \eps^{-N-2s}  \int_{\Omega} \int_{\RR^N} \div_x\left( F_0\left(\frac{y-x}{\eps}\right)(y-x)\right)  [ \vphi(x)-\wvphi(y,y-x)]^2 \d y\d x \\
&\qquad  \quad -\frac 1 2 \eps^{-N-2s}  \int_{\pa \Omega} \int_{\RR^N}F_0\left(\frac{y-x}{\eps}\right)(y-x) \cdot n(x)  [ \vphi(x)-\wvphi(y,y-x)]^2 \d y\d x 
\end{align*}
Now, we split the integrals in $y$ into an integral in $\Omega$ and one in $\RR^N\setminus\Omega$. 
Note that in the second term, when $x\in\pa \Omega$ and $y\in \RR^N\setminus \Omega$ we have 
$ \wvphi(y,y-x)=\vphi(x)$. We thus obtain
\begin{align*}
\int_{\Omega} D_\eps^{2s-1}[\vphi]\cdot \na \vphi \d x 
& = \frac 1 2   \eps^{-N-2s}  \int_{\Omega} \int_{\Omega} \div_x\left( F_0\left(\frac{y-x}{\eps}\right)(y-x)\right)   [ \vphi(x)-\vphi(y)]^2 \d y\d x \\
& \quad +   \eps^{-N-2s}  \frac 1 2    \int_{\Omega} \int_{\RR^N\setminus \Omega} \div_x\left( F_0\left(\frac{y-x}{\eps}\right)(y-x)\right)  [ \vphi(x)-\wvphi(y,y-x)]^2 \d y\d x \\
&  \quad -\frac 1 2 \eps^{-N-2s}  \int_{\pa \Omega} \int_{\Omega}F_0\left(\frac{y-x}{\eps}\right)(y-x) \cdot n(x)  [ \vphi(x)-\vphi(y)]^2 \d y\d x 
\end{align*}
For the second term, we write
\begin{align*}
&    \int_{\Omega} \int_{\RR^N\setminus \Omega} \div_x\left( F_0\left(\frac{y-x}{\eps}\right)(y-x)\right)  [ \vphi(x)-\wvphi(y,y-x)]^2 \d y\d x\\
& \qquad = -  \int_{\Omega} \int_{\RR^N\setminus \Omega} \div_y\left( F_0\left(\frac{y-x}{\eps}\right)(y-x)\right)  [ \vphi(x)-\wvphi(y,y-x)]^2 \d y\d x\\
& \qquad =   \int_{\Omega} \int_{\RR^N\setminus \Omega}  F_0\left(\frac{y-x}{\eps}\right)(y-x)\cdot\na_y  [ \vphi(x)-\wvphi(y,y-x)]^2 \d y\d x\\
& \qquad \quad + \int_{\Omega} \int_{\pa(\RR^N\setminus \Omega)}   F_0\left(\frac{y-x}{\eps}\right)(y-x)\cdot n(y) [ \vphi(x)-\wvphi(y,y-x)]^2 \d y\d x
\end{align*}
where we recall that the vector $n$ points downward.
Using the fact that $v\cdot \na_y \wvphi(y,v)=0$ and $v\cdot\na_v\wvphi(y,v)=0$ for $y\in\RR^N\setminus\Omega$, we see that the first term vanishes and we get (after changing the name of the variables)
\begin{align*}
&    \int_{\Omega} \int_{\RR^N\setminus \Omega} \div_x\left( F_0\left(\frac{y-x}{\eps}\right)(y-x)\right)  [ \vphi(x)-\wvphi(y,y-x)]^2 \d y\d x\\
&= -\int_{\Omega} \int_{\pa \Omega}   F_0\left(\frac{y-x}{\eps}\right)(y-x)\cdot n(x) [ \vphi(y)-\vphi(x)]^2 \d x\d y
\end{align*}

We have thus proved
\begin{align}
\int_{\Omega} D_\eps^{2s-1}[\vphi]\cdot \na \vphi \d x 
& = \frac 1 2   \int_{\Omega} \int_{\Omega} G_\eps(y-x) [ \vphi(x)-\vphi(y)]^2 \d y\d x \nonumber\\
&  \quad -   \int_{\pa \Omega} \int_{\Omega}F_0^\eps(y-x) (y-x) \cdot n(x)  [ \vphi(x)-\vphi(y)]^2 \d y\d x 
\label{eq:aeps}
\end{align}
where
\begin{align*}
F_0^\eps =  \eps^{-N-2s}  F_0\left(\frac{v}{\eps}\right) 
\end{align*}
and
\begin{align*}
G_\eps(v)= -\eps^{-N-2s}   \div_v\left( v F_0\left(\frac{v}{\eps}\right)\right)=\eps^{-N-2s} G(v/\eps), \qquad G(v)= -\div_v(vF_0(v)) .
\end{align*}

We can now use Proposition \ref{prop:d2s} to pass to the limit in the left hand side of \eqref{eq:aeps}.
To pass to the limit in the right hand side of \eqref{eq:aeps}, we proceed as in the proof of Proposition \ref{prop:d2s} using the fact that
$$F_0^\eps \left(\frac{v}{\eps}\right) \sim \frac{\gamma \Gamma(2s)}{|v|^{N+2s}}, \qquad G_\eps \left(\frac{v}{\eps}\right)\sim \frac{2s\gamma \Gamma(2s) }{|v|^{N+2s}}.$$
\end{proof}

Finally we have:
\begin{proof}[Proof of Theorem \ref{thm:evolution}]
The weak solution $\vphi$ given by Theorem \ref{thm:weak} is in $H^s(\Omega)$ and so (by Proposition \ref{prop:D}), $D^{2s-1}[\vphi]\in L^2(\Omega)$. In particular, $\vphi$
satisfies \eqref{eq:weak00}
for all test function $\psi\in \mathcal D(\Omega)$. 
It follows that the equation 
$$ \vphi - \div D^{2s-1}[\vphi]=g$$
holds in $\mathcal D'(\Omega)$. Since both $\vphi$ and $g$ are in $L^2(\Omega)$, we deduce
$$ \mathcal L[\vphi] = \div(D^{2s-1}[\vphi]) \in L^2(\Omega).$$
This implies in particular that the trace $D^{2s-1}[\vphi]\cdot n$ on $\pa\Omega$ is well defined in $H^{-{1/2}}(\pa\Omega)$ and using \eqref{eq:weak00} again, but this time with test functions $\psi\in\mathcal D(\overline \Omega)$, we deduce that
$$ D^{2s-1}[\vphi]\cdot n=0\mbox{ on } \pa\Omega.$$
So if we define the space 
$$ D(\mathcal L)=\{\vphi \in H^s(\Omega)\,;\, \mathcal L[\vphi] \in L^2(\Omega), \quad D^{2s-1}[\vphi]\cdot n=0\mbox{ on } \pa\Omega \},$$
we have proved that the equation 
$$ \vphi  - \mathcal L[\vphi] = g$$
has a unique solution $\vphi\in D(\mathcal L)$ for all $g\in L^2(\Omega)$.

 Theorem \ref{thm:evolution} now follows from Hille-Yoshida theorem.
\end{proof}

\bibliographystyle{siam}
\bibliography{bibliography.bib}

\begin{thebibliography}{10}

\bibitem{AcevesMellet}
{\sc P.~Aceves-Sanchez and A.~Mellet}, {\em Anomalous diffusion limit for a
  linear {B}oltzmann equation with external force field}, Math. Models Methods
  Appl. Sci., 27 (2017), pp.~845--878.

\bibitem{AcevesSchmeiser17}
{\sc P.~Aceves-S\'anchez and C.~Schmeiser}, {\em Fractional diffusion limit of
  a linear kinetic equation in bounded domain}, Kinetic and Related Models, 10
  (2017), pp.~541--551.

\bibitem{Barles14}
{\sc G.~Barles, E.~Chasseigne, C.~Georgelin, and E.~R. Jakobsen}, {\em On
  {N}eumann type problems for nonlocal equations set in a half space}, Trans.
  Amer. Math. Soc., 366 (2014), pp.~4873--4917.

\bibitem{AbdallahMelletPuel11+}
{\sc N.~Ben~Abdallah, A.~Mellet, and M.~Puel}, {\em Anomalous diffusion limit
  for kinetic equations with degenerate collision frequency}, Math. Models
  Meth. Appl. Sci., 21 (2011), pp.~2249--2262.

\bibitem{AbdallahMelletPuel11}
{\sc N.~Ben~Abdallah, A.~Mellet, and M.~Puel}, {\em Fractional diffusion limit
  for collisional kinetic equations: a {Hilbert} expansion approach}, Kinetic
  and Related Models, 4 (2011), pp.~873--900.

\bibitem{BogdanBurdzyChen03}
{\sc K.~Bogdan, K.~Burdzy, and Z.-Q. Chen}, {\em Censored stable processes},
  Probab. Theory Relat. Fields, 127 (2003), pp.~89--152.

\bibitem{BogdanDyda11}
{\sc K.~Bogdan and B.~Dyda}, {\em The best constant in a fractional {H}ardy
  inequality}, Mathematische Nachrichten, 284, pp.~629--638.

\bibitem{Cesbron18}
{\sc L.~Cesbron}, {\em Anomalous diffusion limit of kinetic equations on
  spatially bounded domains}, to appear in Commun. Math. Phys.,  (2018).

\bibitem{CMP}
{\sc L.~Cesbron, A.~Mellet, and M.Puel}, {\em Regularity for a fractional
  neumann boundary value problem}, in preparation,  (2018).

\bibitem{CesbronMelletTrivisa12}
{\sc L.~Cesbron, A.~Mellet, and K.~Trivisa}, {\em Anomalous transport of
  particles in plasma physics}, Applied Mathematics Letters, 25 (2012),
  pp.~2344--2348.

\bibitem{DG}
{\sc J.~Darroz\`es and J.~Guiraud}, {\em G\'en\'eralisation formelle du
  th\'eor\`eme h en pr\'esence de parois}, C. R. Acad. Sci. Paris Ser. A, 262
  (1966), pp.~1368--1371.

\bibitem{Hitch}
{\sc E.~Di~Nezza, G.~Palatucci, and E.~Valdinoci}, {\em Hitchhiker's guide to
  the fractional {S}obolev spaces}, Bull. Sci. Math., 136 (2012), pp.~521--573.

\bibitem{DiNezzaPalatucciValdinoci12}
{\sc E.~DiNezza, G.~Palatucci, and E.~Valdinoci}, {\em Hitchhiker's guide to
  the fractional {Sobolev} spaces}, Bull. des Sci. Math., 136 (2012),
  pp.~521--573.

\bibitem{DiPierroRosotonValdinoci17}
{\sc S.~Dipierro, X.~Ros-Oton, and E.~Valdinoci}, {\em Nonlocal problems with
  {Neumann} boundary conditions}, Revista Matem\'{a}tica Iberoamericana, 33
  (2017), pp.~377--416.

\bibitem{FelsingerKassmannVoigt15}
{\sc M.~Felsinger, M.~Kassmann, and P.~Voigt}, {\em The dirichlet problem for
  nonlocal operators}, Mathematische Zeitschrift, 279 (2015), pp.~779--809.

\bibitem{GuanMa05}
{\sc Q.-Y. Guan and Z.-M. Ma}, {\em Boundary problems for fractional
  {Laplacian}}, Stochastics and Dynamics, 5 (2005), pp.~385--424.

\bibitem{GuanMa06}
{\sc Q.-Y. Guan and Z.-M. Ma}, {\em Reflected symmetric $\alpha$-stable
  processes and regional fractional {Laplacian}}, Probability Theory Relat.
  Fields, 134(4) (2006), p.~649.

\bibitem{Kwasnicki15}
{\sc M.~{Kwa{\'s}nicki}}, {\em {Ten equivalent definitions of the fractional
  {Laplace} operator}}, ArXiv e-prints,  (2015).

\bibitem{LossSloane10}
{\sc M.~Loss and C.~Sloane}, {\em Hardy inequalities for fractional integrals
  on general domains}, Journal of Functional Analysis, 259 (2010), pp.~1369 --
  1379.

\bibitem{Mellet10}
{\sc A.~Mellet}, {\em Fractional diffusion limit for collisional kinetic
  equations: a moments method}, Indiana Univ. Math. J., 59 (2010),
  pp.~1333--1360.

\bibitem{MelletMischlerMouhot11}
{\sc A.~Mellet, S.~Mischler, and C.~Mouhot}, {\em Fractional diffusion limit
  for collisional kinetic equations}, Arch. Ration. Mech. Anal., 199 (2011),
  pp.~493--525.

\bibitem{MelletVasseur}
{\sc A.~Mellet and A.~Vasseur}, {\em Global weak solutions for a
  {V}lasov-{F}okker-{P}lanck/{N}avier-{S}tokes system of equations}, Math.
  Models Methods Appl. Sci., 17 (2007), pp.~1039--1063.

\bibitem{Mischler2}
{\sc S.~Mischler}, {\em On the initial boundary value problem for the
  {V}lasov-{P}oisson-{B}oltzmann system}, Comm. Math. Phys., 210 (2000),
  pp.~447--466.

\bibitem{Mischler1}
\leavevmode\vrule height 2pt depth -1.6pt width 23pt, {\em On the trace problem
  for solutions of the {V}lasov equation}, Comm. Partial Differential
  Equations, 25 (2000), pp.~1415--1443.

\end{thebibliography}

\end{document}